\documentclass[11 pt, reqno]{amsart}
\usepackage{amsfonts}
\usepackage[left= 24mm, right= 24mm]{geometry}
\usepackage{amssymb, amsmath}
\usepackage{natbib}
\usepackage{lscape}
\usepackage{dsfont}
\usepackage{mathrsfs}
\usepackage{mathtools}
\usepackage{bbm}
\usepackage{enumitem}

\usepackage{cancel}
\usepackage{centernot}
\usepackage{verbatim}
\usepackage{appendix}

\DeclareMathAlphabet{\mathpzc}{OT1}{pzc}{m}{it}
\bibpunct{[}{]}{;}{n}{,}{,}
\newtheorem{te}{Theorem}[section]
\newtheorem{defin}[te]{Definition}
\newtheorem{os}[te]{Remark}
\newtheorem{prop}[te]{Proposition}
\newtheorem{lem}[te]{Lemma}

\usepackage[pdftex,plainpages=false,colorlinks,hyperindex,bookmarksopen,linkcolor=red,citecolor=red,urlcolor=red]{hyperref}

\numberwithin{equation}{section}

\allowdisplaybreaks
\def \l { \left( }
\def \r {\right) }
\def \ll { \left\lbrace }
\def \rr { \right\rbrace }

\def \lq {\left[ }
\def \rq {\right] }

\def \lg {\left\{ }
\def \rg {\right\} }

\def \E {\mathbb{E}}
\def \P {\mathbb{P}}
\def \N {\mathbb{N}}
\def \R {\mathbb{R}}

\def \ef {\mathcal{N}^{\nu}} 
\def \Lam {\mathcal{L}}  

\def \Dom {\text{Dom}}

\def \Xcl { X_{(2)}^{\nu}(t) }
\def \Vcl { V_{(2)}^{\nu}(t) }

\def \Xl { X_{(1)}^{\nu}(t) }
\def \Vl { V_{(1)}^{\nu}(t) }

\def \Xt  { \mathcal{X}_t^{(1)} }
\def \Vt  { \mathcal{V}_t^{(1)} }
\def \X { \mathcal{X} }
\def \V { \mathcal{V} }

\def \sgM { \mathcal{T} } 
\def \opmoduno { \mathcal{Z} } 
\def \opmoddue { \widetilde{\mathcal{Q}} } 

\def \sub { \mathcal{U} } 
\def \sgL { \mathcal{Q} } 

\newtheorem*{assumption*}{\assumptionnumber}
\providecommand{\assumptionnumber}{}
\makeatletter
\newenvironment{assumption}[2]
{%
	\renewcommand{\assumptionnumber}{(\textbf{#1#2})}%
	\begin{assumption*}%
		\protected@edef\@currentlabel{(\textbf{#1#2})}%
	}
	{%
	\end{assumption*}
}

\newcounter{mylabelcounter}
\newcommand{\labelText}[2]{%
#1\refstepcounter{mylabelcounter}%
\immediate\write\@auxout{%
  \string\newlabel{#2}{{1}{\thepage}{{\unexpanded{#1}}}{mylabelcounter.\number\value{mylabelcounter}}{}}%
}%
}

\makeatother

\begin{document}

\large



\title[]{Random Flights and Anomalous Diffusion:  \\ A Non-Markovian Take on Lorentz Processes}

\author[]{Lorenzo Facciaroni}
 \author[]{Costantino Ricciuti}
\author[] {Enrico Scalas}
\author[]{Bruno Toaldo}

\keywords{Lorentz process; Boltzmann-Grad limit; Feller semigroups; fractional equations; anomalous diffusion; diffusive limit}
	\date{\today}
    \subjclass[2020]{60K37; 82C40; 60K50; 60K40}
\thanks{The authors acknowledge financial support under the National Recovery and Resilience Plan (NRRP), Mission 4, Component 2, Investment 1.1, Call for tender No. 104 published on 2.2.2022 by the Italian Ministry of University and Research (MUR), funded by the European Union – NextGenerationEU– Project Title “Non–Markovian Dynamics and Non-local Equations” – 202277N5H9 - CUP: D53D23005670006 - Grant Assignment Decree No. 973 adopted on June 30, 2023, by the Italian Ministry of University and Research (MUR)}

\thanks{The author B. Toaldo would like to thank the Isaac Newton Institute for Mathematical Sciences, Cambridge, for support and hospitality during the programme Stochastic systems for anomalous diffusion, where work on this paper was undertaken. This work was supported by EPSRC grant EP/Z000580/1.}

\thanks{E. Scalas also acknowledges financial support under the 000317\_24\_RICERCA\_UNIV\_2023\_PROG\_MEDI\_SCALAS - RICERCA ATENEO 2023 - SCALAS PROGETTI MEDI.
 The title of the project is "Approximation of stochastic processes by means of sums of random telegraph processes"}

 \thanks{The authors are partially supported by Gruppo Nazionale per l’Analisi Matematica, la Probabilità e le loro Applicazioni (GNAMPA-INdAM)}

\begin{abstract}
We  study Lorentz processes in two different settings. Both cases are characterized by infinite expectation of the free-flight times, contrary to what happens in the classical Gallavotti-Spohn models. Under a suitable Boltzmann-Grad type scaling limit, they converge to non-Markovian random-flight processes. A further scaling limit  yields another non-Markovian process, i.e., a superdiffusion  obtained by a suitable time-change of Brownian motion. Furthermore, we obtain the governing equations for our random flights and anomalous diffusion, which  represent a non-local counterpart for the linear-Boltzmann and diffusion equations. It turns out that these  have the form of fractional kinetic equations in both time and space. To prove such results, we develop a technique based on averaging Feller semigroups. 
\end{abstract}

\maketitle
\tableofcontents

\section{Introduction}

\subsection{Context and vision}
A Lorentz process is a model for the motion of a particle among randomly located scatterers \textcolor{black}{also known as obstacles}. It was \textcolor{black}{orginally} used to describe the transport of electrons through a conductor (see \cite{Lorentz}).
This process is difficult to handle, mostly due to the fact that the particle may collide with the same scatterer more than once, \textcolor{black}{and this} produces a memory effect. 
So, one can \textcolor{black}{try and} approximate this process \textcolor{black}{using} more tractable processes. 

The first to tackle this problem rigorously was Gallavotti \cite{Gallavotti}.  In his model, of which we here consider the 3-dimensional version, the scatterers are hard spheres of radius $R$. Their centers are randomly distributed in space according to a Poisson point process with intensity $\rho$, and the particle interacts with the scatterers by means of elastic collisions.  Each free flight time has a marginal exponential distribution, hence it is unlikely that long flights occur.

Intuitively, the memory effect  due to recollisions tends to vanish if $R\to 0$. Indeed, Gallavotti studied this system in the Boltzmann-Grad\textcolor{black}{, also known as low density,} limit: this is the limit in which $R\to 0$ and simultaneously $\rho \to \infty$, in such a way that the free flight time of the particle maintains a finite expectation. Under this limit, the Lorentz process converges to a Markovian random flight process, having independent, exponentially distributed free flight times. \textcolor{black}{Moreover}, the position-velocity process is governed by the \textcolor{black}{following} linear Boltzmann equation
\begin{align}
\frac{\partial}{\partial t} f(x,v,t)= -cv\cdot \nabla _x f(x,v,t)+ \lambda \int _{S^2} (f(x,v',t)-f(x,v,t) ) \mu (dv')
\label{eq. Boltzmann introduzione}
\end{align}
where $\mu$ denotes the uniform distribution of the unit sphere $S^2$, $c$ is the speed of the particle, $\lambda$ is the mean number of direction changes per unit time, \textcolor{black}{and $f(x,v,t)$ is the joint probability density of $X$ (position) and $V$ (velocity) at time $t$}.
After the Boltzmann-Grad limit, a further scaling limit  ($c\to \infty$, $\lambda \to \infty$ such that $c^2/\lambda \to 1$) leads to a diffusion process \textcolor{black}{governed} by the heat equation
\begin{align}
\frac{\partial}{\partial t}p(x,t)= \frac{1}{2} \Delta _x p(x,t),
     \label{eq. diffusione introduzione}
\end{align}
\textcolor{black}{where $p(x,t)$ denotes the probability density of $X$ at time $t$}.
Such a diffusion limit is based on \textcolor{black}{central-limit-theorem} arguments, \textcolor{black}{given that} the free flight times have finite mean and variance.
This model was generalized by Spohn \cite{Spohn}, where more general probability distributions for the scatterers are allowed \textcolor{black}{(here the scatterers are  potentials with finite range $R$), and} weak convergence in path space is proved. Other developments have been made, e.g.,  in \cite{Basile1, Basile2, Boldrighini, Golse}.
A further progress is due to \cite{Toth1, Toth2}, where the Boltzmann-Grad and the diffusion limit are taken simultaneously.

\textcolor{black}{
We stress that all the above mentioned models have a feature in common with the original Gallavotti's work: the limiting processes are Markovian and non-Markovian limits are not considered. }

\textcolor{black}The study of Lorentz models with  non-Markovian scaling limits is a difficult challenge. One of the main problem to address is the derivation of the governing equations, which, in general, are not expected to be differential equations of the first order in time, \textcolor{black}{such as} \eqref{eq. Boltzmann introduzione} and \eqref{eq. diffusione introduzione}. Indeed, according to Feller's theory, Markov processes are governed by their \textcolor{black}{first-order-time-derivative infinitesimal} generator equations, \textcolor{black}{whereas} non-Markov processes are generally \textcolor{black}{related} to non-local  equations. 

Recently, some \textcolor{black}{Lorentz models} with non-Markovian limits have appeared in the literature. In these models, the gas is immersed in an external force field. It is because of this  force field that the position-velocity process is \textcolor{black}{no longer} Markovian after the kinetic limit. \textcolor{black}{The interested reader can e.g. see} \cite{DesvillettesRicci}. \textcolor{black}{One can also consult}
\cite{NotaSaffirio, NotaSaffirioSimonella} and all the references therein, where the authors consider  a bi-dimensional Lorentz gas under the action of a constant magnetic field. \textcolor{black}{In those paper,} a generalization of the Boltzmann equation, which is non-local in the time variable, appears (see e.g. formula (1.3) in \cite{NotaSaffirioSimonella}); also  a non-local heat equation where the diffusion coefficient is replaced by an operator \textcolor{black}{written} in terms of the Green-Kubo formula (see formula (1.8) in \cite{NotaSaffirio}). 

\textcolor{black}{We want to study the aspects of non-Markovianity} from a different perspective, without  \textcolor{black}{applying} external forces. \textcolor{black}{We summarize our research plan} as follows. 
\vspace{0.1cm}

\textcolor{black}{
\textit{We are interested in  Lorentz models with non-Markovian scaling limits, where the origin of non-Markovianity just lies in the probability distribution of the obstacle centers.}}
\vspace{0.1cm}

As we shall see, our paper contributes to this \textcolor{black}{goal} by studying a particular \textcolor{black}{obstacle} distribution, \textcolor{black}{giving} rise to a special class of non-Markov processes. Broader classes are of \textcolor{black}{great} interest, but we defer \textcolor{black}{their} study to future work.

Our work is connected to Spohn's results (see \cite{Spohn} and also Section 4 in \cite{Spohn_review}). In particular, Spohn presented necessary and sufficient conditions to be obeyed by the obstacle distribution so that the limit random flight is Markovian. Informally speaking, these conditions mean that, in the low density limit, the rescaled number of obstacles centres in any Borel set converges in probability to a constant, and, consistently, the rescaled number of obstacles centres lying in disjoint sets are independent. Our obstacle distribution does not meet Spohn's conditions (see the discussion in Section \ref{Sec: The Boltzmann-Grad limit modello 1}) and a non-Markovian limit emerges.

\textcolor{black}{From the  physical and applicative point of view,  our work is motivated by two points:}

\begin{enumerate}
    \item In the existing literature, the Lorentz processes among random obstacles  are approximated by diffusion processes. Limiting anomalous diffusions (in the sense explained later) are not considered.

\item Under the kinetic limits presented in the existing models, the expected value of the free flight times remains finite (either  strictly positive or zero).
Cases in which the flight times have  infinite expectation are  excluded (we here refer to  flight times that are  almost surely finite, i.e. they have a proper probability distribution, but their expected value is infinite). 
\end{enumerate}

On point (1), we stress that applied scientists show up more and more frequently that classical models (based on Brownian motion, It\^o diffusions and Markov processes) are too regular and not suitable to cover stochastic systems exhibiting complicated features alterating the diffusive behaviors. In these cases the mean square displacement of the observed variable $X(t)$, i.e., the quantity $\mathds{E} \left\| X(t) -X(0)\right\|^2$, where $X(t)$ denote the position of a particle at time $t$, spreads faster or slower than linearly, as $t \to +\infty$; the linear behavior is, indeed, the typical “diffusive” one, related to Brownian motion and the heat equation. These deviations arise in several contexts. The earliest signs of anomalous diffusion were documented by Richardson in his landmark 1926 study on turbulence \citep{richardson1926}. Later, more systematic theoretical approaches appeared in the 1970s, especially through Scher and Montroll’s investigation of charge transport in amorphous semiconductors using the continuous time random walk (CTRW) model \citep{scher1975}.
Subsequently, experimental evidence for anomalous diffusion has been recorded across a variety of settings, including porous materials, polymer chains, cellular systems, turbulent flows, and biological contexts \citep{barkai2000, golding2006, metzler_fd, sokolov2005}. Subdiffusive behavior is commonly observed in polymer dynamics, protein transport, and NMR experiments involving disordered media, while superdiffusive dynamics are typically associated with Lévy statistics found in turbulent and biological systems \citep{bouchaud1990, weiss1994, zumofen1993}.
A very popular approach for modeling these anomalous diffusions is based on CTRWs, L\'evy walks and L\'evy flights (see, e.g., \cite{balescu2005, bardou2001, havlin1987, isichenko1992, metzler_fd, metzler2004, montrollshlesinger1984} and references therein). L\'evy walks, in particular, are special random flights (with independent flight times); see \cite{zaburdaev} for a review.

However, it turns out that a unifying theory connecting anomalous diffusions and mechanical  processes of Lorentz type (with deterministic Hamiltonian and random obstacles) is still not available. A pioneering contribution in this context is \cite{toth_super}, where a connection is established for the periodic Lorentz gas when the obstacles' position is deterministic.
Note that for further crucial results on the periodic Lorentz gas, the reader can also consult \cite{golse1, golse3, golse2, anmath}. We also refer to
\cite{Marklof} where the authors use a very general approach, with a large class of deterministic scatterer configurations giving rise to non-Markovian limits.

As far as point (2) is concerned, we stress that cases with infinite mean flight times are \textcolor{black}{important} from a physical point of view, since they are just connected to a class of the above mentioned anomalous diffusion models. On this point there  is a wide literature concerning CTRWs and random flights where heavy-tailed waiting times belong to the domain of attraction of a $\nu$-stable law, $\nu \in (0, 1)$.  See, e.g. \cite{MeerAnnals, MeerJap, MeerPhysRevE, MeerSpa, strakaSpa} for continuous-time random walks and \cite{RicciutiToaldo} for a recent random flight model. These processes and their  anomalous diffusive limits fall \textcolor{black}{within} the theory of semi-Markov processes \cite{meerstraannals} and they are crucial \textcolor{black}{to understand and model} anomalous diffusive phenomena related to fractional dynamics (see, e.g., \cite{metzler_fd} and references therein). Random flights with flight times having diverging moments have been extensively considered in the physics literature: as previously mentioned, these processes are called L\'evy walks \cite{zaburdaev} and they have plenty of applications \textcolor{black}{to} modeling anomalous phenomena: e.g., in the context of interacting particles \cite{fedotov1}, intracellular transport \cite{fedotov2}, blinking nanocrystals \cite{barkai1, barkai2}, classical chaos and nonlinear hydrodynamics \cite{klafter1996, Shlesinger1999} and many others (see \cite{klaftersokolov, zaburdaev} and references therein).

One might ask, therefore, whether random walks/flights with infinite mean flight times and the related anomalous diffusions can be connected to mechanical models of Lorentz type.

\subsection{Summary of main results}
In this paper we present two Lorentz process models. For both of them, the limit random flights and their related anomalous diffusion, together with their governing equations, are new. Such processes go beyond the semi-Markov setting and a useful tool in order to deal with them is the recent theory of the so-called para-Markov processes, \textcolor{black}{based on exchangeability} (see \cite{Facciaroni}).

The first model, discussed in Section \ref{Sec: motion among random obstacles (model 1)}, is based on a distribution of obstacles that we call Mittag-Leffler point process. The reason \textcolor{black}{for} this name is that the free flight times have a so-called  Mittag-Leffler distribution,  a heavy tailed distribution belonging to the domain of attraction of a stable law, with interesting analytical properties; its density function has a power law decay $\sim C t^{-\nu-1}$, where $\nu\in (0,1)$ and $C>0$, hence the expectation is infinite. 
In this setting it is likely that long flights occur, contrary to what happens for the Poissonian case. 
We study this system under a limit of Boltzmann-Grad type (see Theorem \ref{Theorem BG boltzmann grad}).

As discussed before, this distribution of obstacles does not satisfy Spohn's conditions and thus a non-Markovian limit is obtained. Moreover, in our model,  the free flight time of the particle has infinite mean both before and after the limit. 
The limiting non-Markovian random flight presents a long memory tail, in the sense that all the flight times are dependent; despite this, the process is  easy to handle, since it is possible to explicitly write the joint distribution of its flight times, which happens to be of Schur constant type. In particular, let $\{\opmoduno_t\}_{t\geq0}$ be the evolution operator associated to the position-velocity process. Then we prove the following Duhamel expansion (see Theorem \ref{Model 1: teorema su Duhamel})
\begin{align}
        \opmoduno_t h  = \, & \mathcal{M}_\nu (-\lambda ^\nu t^\nu)P_th  \notag \\ &+ \sum _{n=1}^\infty (-\lambda )^n \mathcal{M}_\nu ^{(n)}(-\lambda ^\nu t^\nu)  \int _{0< \tau _1< \tau _2< \dots < t}P_{t-\tau _n}L  P_{\tau _n-\tau _{n-1}}L \dots LP_{\tau _1}h\,  d\tau _1 \dots d\tau _n \label{thisthis}
    \end{align}
    where $L$ is the scattering operator, $P_t$ is the translation operator and $\mathcal{M}_\nu^{(n)}(\cdot)$ denotes the $n$-th derivative of the Mittag-Leffler function. Note that for $\nu=1$ formula \eqref{thisthis} reduces to the classical Duhamel expansion of the Boltzmann semigroup.

After the Boltzmann-Grad limit, a further scaling limit leads to an anomalous diffusion process, exhibiting self-similarity and also super-diffusive behaviour, i.e. the particles spread faster than for Brownian motion. We call such a process {\em the Mittag-Leffler anomalous-diffusion process}. Specifically, in Section \ref{Diffusion model 1}, we show that its density $q(x,t)$ is the fundamental solution of a fractional diffusion equation of the form
\begin{align*}
    \frac{\partial ^\nu}{\partial t ^\nu} q(x,t)= - \frac{1}{2^\nu} (-\Delta )^\nu q(x,t),
\end{align*}
where the non-local time derivative is defined in Equation \eqref{deffrac} and $ -  (-\Delta )^\nu$ is the fractional Laplacian. Nevertheless, being $c$ the speed of the random flight and $\lambda$ the parameter of each flight time,  under the scaling $c\to \infty,\ \lambda \to \infty,\ c^2/\lambda = 1 $ we get
\begin{align*}
    \lg X_{(1)}^{\nu}(t)\rg \overset {fdd}{\longrightarrow}\{W_t\},     
\end{align*}
where $\lg X_{(1)}^{\nu}(t)\rg_{t\geq0}$ is the position random position in the first model and $\lg W_t\rg_{t\geq0}$ is the Mittag-Leffler anomalous-diffusive process. The convergence is in the sense of finite dimensional distributions.

The same anomalous diffusion is obtained in the second model (studied in Section \ref{section model 2}), where the distribution of obstacles is Poissonian (as in \cite{Gallavotti}), but the speed of the particle is heavy tailed distributed.  
A Boltzmann-Grad limit leads to another non-Markovian random flight process with speed and flight times depending on each other. 
In particular, let $\lg \opmoddue_t \rg_{t\geq0}$ be the evolution operator associated to the position-velocity process. Then, in Theorem \ref{Model 2: teorema equazione Boltzmann frazionaria e uguaglianza in distribuzione}, we show that the map $t\longmapsto \opmoddue_t h$ is the unique solution of the fractional Boltzmann equation
\begin{align*}
    \frac{\partial ^\nu}{\partial   t^\nu} g(t) = -\bigg (-c\, v\cdot \nabla  -\lambda (L-\mathcal{I}) \bigg )^\nu g(t)  ,
\end{align*}
for $h$ satisfying appropriate conditions and for $g(0) = h$. The fractional power in the right hand is rigorously  defined in Equation \eqref{Isotrpoic: fractional Boltzmann operator modello 2}.
As said before, a second  scaling limit gives the same super-diffusion of the first model.

In order to rigorously get the above governing equations, we face a problem which is in itself interesting from a probabilistic and analytical point of view and serves as a tool, in this context. Let $\{T_t\} _{t\geq 0}$ be a Feller semigroup on a Banach space $\mathfrak{B}$, with generator $(G, \textrm{Dom}(G))$. Moreover,  consider the Bochner integral $Q_t= \int _0^\infty T_s \, l(ds)$, where $l(ds)$ denotes the so-called Lamperti distribution (see e.g. \cite{Lamperti_James}). This is  a distribution on $(0,\infty)$  with power law decay $t^{-\nu-1}$, $\nu \in (0,1)$. The family $\{Q_t\} _{t\geq 0}$ defines a mixture of Feller semigroups. We prove in Theorem \ref{teorema equazione} that, for any $h\in \mathrm{Dom} (G)$,  the function $t \mapsto Q_th$ solves the abstract Cauchy problem
\begin{align}
\frac{\partial ^{\nu}}{\partial t^{\nu}}g (t) = -(-G)^\nu g(t), \qquad g(0)=h, \label{equazione modello}
\end{align}
i.e. a fractional kinetic equation exhibiting fractional powers of operators. Such equation reduces to the well-known  governing equation of Markov processes for $\nu = 1$. 
So, we see  that this kind of fractional kinetic equations emerge after kinetic limits of our models. They generalize equations \eqref{eq. Boltzmann introduzione} and \eqref{eq. diffusione introduzione} by introducing the $\nu$-powers of operators.

\section{Preliminary notions} \label{Basic notions and preliminary results}

In the present paper we shall use results from very different theories, e.g., we shall combine tools from kinetic theory, the theory of non-Markov processes and non-local equations. Hence, for the reader convenience, we summarize the main facts in few pages.
\subsection{Isotropic transport processes}
In this paper we consider a class of transport processes, also called random flights, defined as follows.
Let a particle start at a point $x\in \mathbb{R}^d$ and move along a unit vector $v_0$ with random speed $\mathcal{C}$, for a random waiting time $J_1$.
After the time interval $J_1$, the particle undergoes scattering and continues its motion along a new unit vector $v_1$, again with speed $\mathcal{C}$, for another random waiting time $J_2$. This process repeats indefinitely.
Here $\mathcal{C}$ is assumed to be either a positive constant or a positive random variable.
The directions $\{v_i,\ i\geq0\}$ are assumed to be i.i.d. random vectors, with uniform law on the unit sphere
\begin{align*}
    S^{d-1}= \{ v\in \mathbb{R}^d: |v|_e^2=1\},
\end{align*}
where $|\cdot|_e$ denotes the Euclidean norm. 
Moreover, we assume that the unit vectors $\{v_i,\ i\geq0\}$ are independent of the waiting times $\lg J_n,\ n\geq1 \rg$ and of the speed $\mathcal{C}$. On the other hand, $\mathcal{C}$ and $\lg J_n,\ n\geq1 \rg$ may be dependent.
Let $\tau_n := \sum_{i = 1}^n J_i$ be the instant of the $n$-th change of direction for $n\geq1$, with the convention $\tau_0 : =0$,  and let $\mathcal{N}=\{\mathcal{N}_t,\ t\geq0\}$ be the process which counts the number of changes up to time $t$, i.e.
\begin{align*}
    \mathcal{N}_t := \max \{n \geq 0: \tau _n \leq t\}.
\end{align*}
At time $t\geq0$, the particle has unit velocity vector $V_t$ and is located at the position $X_t$.
To summarize the above description, we now give the rigorous definition of the joint process $(X_t, V_t)$.

\begin{defin} \label{Definizione isotropic transport process generica}
    Let $\l \Omega, \mathcal{A}, \mathds{P}^{(x,v)} \r,\ v\in S^{d-1},\ x\in\R^d$, be a family of probability spaces.
    Let $ \{J_n\}_{n \geq 1} $, $\{v_i\}_{i\geq0}$ and $\mathcal{C}$ be random variables as defined above. Let $\mathcal{N}$ be defined as above. 
    The joint process $ \lg \l X _t,  V _t\r ,\ t\geq0\rg$, such that
    \begin{align} \label{Isotropic: definizione velocità transp process}
        V _t &:= v_{\mathcal{N}_t} \\   
    X_t  &:= X_0 + \mathcal{C}\int_0^t V_s  ds 
     = X_0+ \mathcal{C} \sum _{i=1}^{\mathcal{N}_t} J_i v_{i-1}+ \mathcal{C}\ (t-\tau _{\mathcal{N}_t})\  v_{\mathcal{N}_t}, \label{Isotropic: definizione posizione transp process}
    \end{align}
     where $\mathds{P}^{x,v}(X_0= x ,v_0 = v)=1$, is said to be the isotropic transport process.
\end{defin}

In practice, defining such a transport process requires specifying only the joint law of the vectors $( \mathcal{C},  J_1, J_2, \cdots, J_n)$ for any $n$, from which the joint distribution for the flight times $\lg J_n,\ n \geq1  \rg$ and the distribution of the speed $\mathcal{C}$ directly follow.

In Section  \ref{Markovian random flight} we recall some known facts on the Markovian case, i.e., the case where the $J_n$s are i.i.d. exponential r.v.'s and $\mathcal{C}$ is a positive constant chosen independently on the $J_n$s. Then, in Sections \ref{Subsection: Random flight Model 1} and \ref{Subsection: Random flight Model 2}, we shall define two non-Markovian models, where the waiting times are stochastically dependent and have infinite mean. Furthermore, their limit behaviour will be studied in Sections \ref{Diffusion model 1} and \ref{Boltzmann-Grad approximation and anomalous diffusion model 2}.

\subsubsection{The Markovian random flight and its diffusive limit} \label{Markovian random flight}

Let the waiting times $\{J_n, n\geq 1\}$ be i.i.d. exponential random variables with mean $\lambda ^{-1}$ and let the speed $\mathcal{C}$ be a positive constant, say $\mathcal{C}=c$ almost surely. Then the resulting transport process is the Markovian isotropic transport process, as studied, for example, in \cite{kolesnik, Monin, Orsingher1, Stadje, Watanabe}; an analogous construction holds in non Euclidean spaces (see e.g. \cite{Orsingher2,pinsky}). In this setting, the counting process $\mathcal{N}$ is a Poisson process with intensity $\lambda$ and $\tau_n$ follows a Gamma distribution of parameters $n,\lambda$, also called Erlang distribution. Here the parameter $\lambda$ can be interpreted as the mean number of direction changes per unit time.

In this case, we shall denote the random flight process of Definition \ref{Definizione isotropic transport process generica} by 
\begin{align*}
   \lg \l X^{(c, \lambda)}_t, V^{(c, \lambda)}_t \r,\ t\geq 0 \rg ,
\end{align*}
where the superscript $(c,\lambda)$ indicates that the process depends on the two parameters $c$ and $\lambda$.
Here $\lg V^{(c, \lambda)}_t,\ t\geq0 \rg$ is    a  Markov process on $S^{d-1}$ and  $\lg \l X^ {(c, \lambda)}_t, V^{(c, \lambda)}_t \r,\ t\geq0 \rg$ is a Markov process on $\mathbb{R}^d \times S^{d-1}$. 

Let us consider initial data $(x,v)\in \R^d\times S^{d-1}$. 
Moreover, consider the family of operators $\lg \sgM_t^ {(c, \lambda)} \rg_{t\geq0}$ defined by
\begin{align}
    \sgM_t^ {(c, \lambda)}h (x,v) := \mathds{E}^{(x,v)}\lq h \l X^ {(c, \lambda)}_t, V^{(c, \lambda)}_t \r \rq \qquad h\in C_0\l \R^d \times S^{d-1}\r  \label{semigruppo Boltzmann}\end{align}
where $\mathds{E}^{(x,v)}$ denotes the expectation under $\mathds{P}^{x,v}$.
Such a family defines a strongly continuous contraction semigroup on the space $C_0 \l \R^d \times S^{d-1}\r $ endowed with the sup-norm $||\cdot||$, i.e. a Feller semigroup (see \cite[Lemma 2.1]{Watanabe} or, for a very general approach, \cite{pinsky}). 

 Let us consider the semigroup $\lg e^{-\lambda t}P_t \rg_{t\geq0}$ where
 \begin{align}
     P_t h(x,v) := h(x+vct,v) \qquad h \in C_0 \l \R^d \times S^{d-1} \r. \label{Random flight 1: definizione operatore Pt si traslazione}
 \end{align}
 Then $\lg e^{-\lambda t } P_t \rg_{t\geq0}$ is strongly continuous and is generated by the operator $cv \cdot \nabla_x - \lambda \mathcal{I}$, where $\mathcal{I}$ denotes the identity operator, on the domain 
\begin{align}
    \text{D}=\ll h \in C_0 \l \R^d \times S^{d-1} \r: c v \cdot \nabla_x h \in C_0 \l \R^d \times S^{d-1} \r \rr. \label{Isotropic: dominio del generatore del semigruppo markoviano}
\end{align}
Furthermore, let us consider  the scattering operator $L$ defined by

\begin{align}
    Lh(x,v) := \int _{S^{d-1}} h(x,v')\mu (dv'), \qquad h\in C_0 \l \R^d \times S^{d-1} \r, \label{Random flight 1: definizione operatore L di scattering}
\end{align}
 where $\mu$ denotes the uniform distribution on $S^{d-1}$. It is useful to see how \cite[Theorem 1.9.2]{Kolobook} applies here, giving the following result which is known in the literature.

\begin{te} \label{Isotropic: teorema rappresentazione semigrouppo Boltzmann}
The operator $\sgM_t^ {(c, \lambda)}$ defined in Equation \eqref{semigruppo Boltzmann} has the following representation on $D$ 
     \begin{align}
         \sgM_t^ {(c, \lambda)} = e^{-\lambda t} P_t + \sum _{n=1}^\infty \lambda ^n e^{-\lambda t}   \int _{0< \tau _1< \tau _2< \dots <\tau_n < t} P_{t-\tau _n}L   P_{\tau _n-\tau _{n-1}}L \dots L P_{\tau _1}\,  d\tau _1 \dots d\tau _n.
        \label{Duhamel}
    \end{align}
    Moreover, for $h \in \text{D}$, one has that $t \mapsto \sgM_t^ {(c, \lambda)}h$ is the unique (bounded) solution
to the linear (abstract) Boltzmann equation on $C_0 (\mathbb{R}^{d} \times S^{d-1})$
    \begin{align} 
        \frac{\partial}{\partial t} g(t)=  c\, v\cdot \nabla _x  g(t)+\lambda (L-\mathcal{I}) g(t), \label{Boltzmann equation}
\end{align}
under the initial condition  $g(0)= h$. 
    
\end{te}
\begin{proof}
     First, we observe that the operator $L$ is bounded on $C_0 \l \R^d \times S^{d-1} \r$. Moreover, $cv\cdot \nabla_x - \lambda \mathcal{I}$ generates the semigroup $\lg e^{-\lambda t}P_t \rg_{t\geq0}$, which is strongly continuous. Hence, we can apply \cite[Theorem 1.9.2, (i)]{Kolobook} to say that 
    \begin{align}
        G^ {(c, \lambda)} :=  c\, v\cdot \nabla _x - \lambda \mathcal{I} +\lambda L, \label{Boltzmann generator}
    \end{align}
    with domain $D$,  is the generator of a strongly continuous semigroup on $C_0 \l \R^d \times S^{d-1} \r$, which has the representation
    \begin{align*}
        \Phi_t = e^{-\lambda t} P_t + \sum _{n=1}^\infty \lambda ^n e^{-\lambda t}   \int _{0< \tau _1< \tau _2< \dots <\tau_n < t} P_{t-\tau _n}L   P_{\tau _n-\tau _{n-1}}L \dots L P_{\tau _1}\,  d\tau _1 \dots d\tau _n.
    \end{align*}
    Moreover, \cite[Theorem 1.9.2, (ii)]{Kolobook} guarantees that $\Phi_th$ is the unique solution of
    \begin{align}
        \Phi_t h \, = \, e^{-\lambda t}P_t h \, + \, \int_0^t \lambda e^{-\lambda (t-s)} \Phi_s \, L \, P_{t-s} \, h \, ds.
        \label{inteq}
    \end{align}
    We now observe that $ || \sgM_t^{(c, \lambda)}h|| \leq\left\| h \right\|$ and, by a standard conditioning argument, using the first renewal time, one has 
    \begin{align*}
        \sgM_t^{(c,\lambda)} h(x,v) 
        &= e^{-\lambda t}h(x + cvt, v) + \int_0^t \int_{S^{d-1}} \lambda e^{-\lambda (t-s)}    \sgM_s^{(c,\lambda)} h(x + cv (t-s), \omega) \mu(d\omega) \, ds,
    \end{align*}
    i.e. $t \mapsto \sgM_t^{(c, \lambda)}h$ satisfies Equation \eqref{inteq}. By the uniqueness of solution, we get the thesis.
\end{proof}

We finally observe that the sum in Equation \eqref{Duhamel} converges in the sup-norm and can be explicitly written as 
\begin{align}
    & \sgM_t^ {(c, \lambda)} h(x,v) =  e^{-\lambda t} h(x+cvt,v) + \\\notag  
    &+ \sum _{n=1}^\infty \lambda ^n e^{-\lambda t}   \int _{0< \tau _1< \tau _2< \dots <\tau_n < t} \int _{\l S^{d-1} \r^{n}}\,  h(x_t, \omega_{n})\, d\tau _1 \dots d\tau _n \, \mu (d\omega_1)\dots \mu (d\omega_n)
\end{align}
where
\begin{align*}
    x_t:= x+ c\sum _{j=1}^{n} (\tau_j - \tau_{j-1}) \omega_{j-1} + c(t-\tau _{n})\, \omega_{n}, \qquad \omega_0 := v, 
\end{align*}
and $v_t = \omega_n$ for $t\in [\tau_n, \tau_{n+1})$.

Another useful result we will use is the following explicit formula for the mean squared  displacement of the Markovian random flight (details can be found in \cite{Garra}). Denote by $\mathds{P}^{x,\mu}$ the probability measure
\begin{align}
  \mathds{P}^{x,\mu} (\cdot) \coloneqq  \int_{S^{d-1}} \mathds{P}^{(x,v)} (\cdot) \, \mu(dv)
\end{align}
where $\mu$ is the uniform probability measure on the $d-1$-dimensional sphere and by $\mathds{E}^{x,\mu}$ the corresponding expectation. Then
\begin{align}
  \mathds{E}^{x,\mu} \left | X_t^{(c,\lambda)} - x\right |^2_e = \frac{2c^2}{\lambda ^2} \l \lambda t - 1 + e^{-\lambda t} \r.\label{Markovian: mean squared  displacement random flight}
\end{align}

A remarkable fact is that the Markovian isotropic transport process is a finite speed and finite rate approximation of Brownian motion. Indeed, for  Brownian motion both the speed of the particle and the rate of direction changes are infinite; but nevertheless, in many practical applications, Brownian motion is used to model random motions with large but finite speed and rate. This is allowed by the following result.
\begin{prop} \label{limite diffusivo classico} 
Let $B = \{B_t,\ t\geq0\}$ be a $d$-dimensional Brownian motion, $d\geq1$ under a probability measure $\mathbb{P}^x$, s.t. $\mathbb{P}^x(B_0 = x)=1$. Then, under the scaling limit 
\begin{align*}
c\to \infty \qquad \lambda \to \infty \qquad \frac{c^2}{\lambda} = D>0
\end{align*}
we have that 
\begin{align*}
    \lg X_t^{(c,\lambda)}\rg \overset{\text{fdd}}{\longrightarrow} \lg B_{Dt} \rg .
\end{align*}
\end{prop}
The reader can consult \cite{kolesnik, pinsky, Watanabe} and references therein for details and an improved version of this statement showing weak convergence on the space of continuous functions with the uniform topology.

\subsection{Para-Markov chains} \label{Para-Markov chains}
In this section we recall the definition of a class of non-Markovian processes, called para-Markov \cite{Facciaroni}, as these processes will be a useful tool in the following. These are right continuous processes on a countable state space whose waiting times between jumps have a suitable stochastic dependence,  creating a long memory tail in the evolution. They are proved to be equal in distribution to a continuous-time Markov chain whose time parameter is randomly scaled, hence the name para-Markov.

For a positive $\nu$, we shall indicate with $ \mathcal{M}_\nu$ the one-parameter Mittag-Leffler function, defined by
\begin{align}
    \mathcal{M}_\nu(x):= \sum _{k=0}^\infty \frac{x^k}{\Gamma (1+\nu k)} \qquad x\in \mathbb{R} .\label{Mittag Leffler function}
\end{align}
In the paper we shall use the notation 
\begin{align}
    \mathcal{M}_\nu^{(k)}(-z^{\nu}) :=  \l \frac{d}{dz} \r^k \mathcal{M}_{\nu}(-z^{\nu})  \qquad \qquad k\in\N .\label{Derivative of Mittag Leffler function}
\end{align}
Moreover, we recall that a non-negative random variable $J$ is said to follow a Mittag-Leffler distribution with parameters $\nu\in(0,1]$ and $\lambda \in (0, \infty)$ if it has cumulative distribution function
\begin{align*}
    \P(J \leq t)= 1- \mathcal{M}_{\nu}(-\lambda t^{\nu}), \qquad t \geq 0.
\end{align*}
For $\nu \in (0,1)$, by using the asymptotic properties of the Mittag-Leffler function (see, e.g., \cite{meertoa}), one gets
\begin{align}
    \P(J > t) \sim \alpha t^{-\nu} \qquad \textrm{as} \qquad t\to \infty,
    \label{asymmittag}
\end{align}
being $\alpha>0$ a constant. If follows that $J$ has infinite expectation.
Instead, for $\nu=1$ we have $\mathcal{M}_1(x)=e^x$ and $J \sim \text{Exp}(\lambda)$.

\begin{defin} \label{Definition of  para-Markov chains}
Let $Y = \{Y_n,\ n\in\N\}$ be a discrete time Markov chain, on some probability space with a probability measure $P$, and a finite or countable state space $\mathcal{S}$. For $\nu \in (0,1]$ and $\lambda : \mathcal{S} \to (0,\infty)$, let $\lg J_n,\ n\geq 1\rg$ be a sequence of non-negative random variables, such that,  $\forall n\geq1$,
\begin{align} \label{conditional dependence structure for para-Markov}
    & P(J_1>t_1, \dots, J_n>t_n |Y_0=y_0, \dots, Y_{n-1}=y_{n-1}) = \mathcal{M}_\nu \left ( -\left (\sum _{k=1}^{n} \lambda (y_{k-1})t_{k} \right )^\nu \right ) 
\end{align}
 where $t_k\geq0,\ k \in\{1,\ldots,n\}$. Let $T_n:= \sum_{k=1}^nJ_k$ and $T_0 := 0$.
A continuous-time process $X^{\nu} = \{X_t^{\nu},\ t \geq 0\}$ such that 
\begin{align*}
    X_t^{\nu} = Y_n \qquad \qquad t\in[T_n, T_{n+1}),\ n\in\mathbb{N}
\end{align*}
 is said to be a  para-Markov chain.
\end{defin}
From the above definition we have that each waiting time follows a Mittag-Leffler distribution:
$$
 P(J_k>t_k | Y_{k-1}=y_{k-1})  = \mathcal{M}_\nu \l  - \left (\lambda (y_{k-1})t_{k} \right )^\nu \r.
$$
Moreover, note that for $\nu=1$ the joint survival function in Equation (\ref{conditional dependence structure for para-Markov}) is factorized as
\begin{align} 
 P(J_1>t_1, \dots, J_n>t_n |Y_0=y_0, \dots, Y_{n-1}=y_{n-1}) = e^{ -\sum _{k=1}^{n} \lambda (y_{k-1})t_{k} },
\end{align}
i.e. the waiting times are conditionally independent, exponential random variables and hence $X^{1}$ is a homogeneous, continuous-time Markov chain.

From \cite{Facciaroni}, it is known that a  para-Markov chain $ X^{\nu} $ is equivalent, in terms of finite-dimensional distributions, to a Markov chain $ X^1 $ through a random time-change, as reported below. Specifically, this is achieved by substituting the time parameter $ t $ with $ t\Lam $, where $ \Lam $ has a Lamperti distribution, which is defined as follows (see e.g. \cite{Lamperti_James}).

\begin{defin}\label{Definition of Lamperti distribution}
    A non-negative random variable $\mathcal{L}$, on some probability space with probability $P$, follows a Lamperti distribution of parameter $\nu\in(0,1]$ if its Laplace transform is given by
\begin{align}
    E\left[e^{-\eta \mathcal{L}}\right] = \mathcal{M}_{\nu}\left(-\eta^{\nu}\right), \qquad \eta \geq 0. \label{lamperti laplace}
\end{align}
\end{defin}
For $\nu =1$ we have $E\lq e^{-\eta \mathcal{L}} \rq =e^{-\eta}$ which means $\mathcal{L}\overset{a.s.}{=}1$, while for $\nu \in (0,1)$ the variable $\mathcal{L}$ is absolutely continuous, with density given by
\begin{align}
    \mathpzc{l}(dy) := \frac{\sin \pi \nu}{\pi} \frac{y^{\nu-1}}{y^{2 \nu} +2y^\nu \cos \pi \nu +1} \mathds{1}_{(0, +\infty)}(y)\ dy. \label{Misura Lamperti distribution}
\end{align}

Consult \cite[Equations (1.1) and (3.3)]{Lamperti_James} with $\theta = 0$ and $\eta = z^{\frac{1}{\alpha}}$ for details. We also recall a property which will be useful later in the article:
\begin{align}
    \mathcal{L} \overset{d}{=} \frac{1}{\mathcal{L}},\label{Lamperti: L = 1/L}
\end{align}
where $\overset{d}{=}$ denotes the equality in distribution. 
\begin{te} \label{Proposition:  para-Makov chains as time-change of Markov chains through Lamperti}
Let us consider a para-Markov chain $X^{\nu}$ as in Definition \ref{Definition of  para-Markov chains}.
Let $\Lam$ be a Lamperti random variable of parameter $\nu\in(0,1]$.
Then we have
    \begin{align*}
        X_t^{\nu} \overset{fdd}{=} X_{\Lam t}^{1} \qquad \qquad \forall t \geq0,
    \end{align*}
    where $\overset{\text{fdd}}{=}$ denotes equality of  finite dimensional distributions.
\end{te}
For a proof, see \cite[Theorem 4]{Facciaroni}.

A corner point is that the para-Markov chain $X^\nu$ is governed by an integro-differential equation. The operator acting on the time variable is the Caputo fractional derivative (see Equation (1.10) in \cite{GorenfloMainardi} or \cite{Baleanu} for a general introduction), defined by
\begin{align}
    \frac{\partial ^\nu }{\partial t^\nu} f(t):= \begin{cases}
     \frac{1}{\Gamma (1-\nu)}\frac{\partial}{\partial t} \int_0^t (f(s)-f(0)) (t-s)^{-\nu} ds ,  & \nu \in (0,1), \\
    f'(t), & \nu =1.
\end{cases}\label{Definizione di Derivata di Caputo}
\end{align}

Let us consider a para-Markov chain $X^{\nu}$ and let us indicate with $G$ the infinitesimal generator of the Markov chain $X^{1}$. Let  $P(t)=[p_{ij}(t)]$ be the transition matrix of $X^{\nu}$. In \cite{Facciaroni} it is proven that in the case of finite states space $\mathcal{S}$, the transition matrix $P(t)$ is the solution of
\begin{align}
    \frac{\partial ^{\nu}}{\partial t^{\nu}} P(t) = -(-G)^{\nu} P(t) \label{Governing equation para markov}
\end{align}
with initial condition $P(0) = I$, being $I$ the identity matrix.
We stress that for $\nu =1$, Equation \eqref{Governing equation para markov} formally reduces to the Kolmogorov backward equation which governs a Markov chain.

\subsection{The para-Markov counting process}\label{Exchengeable Fractional PP}

We here recall the definition of a counting process,  which will be useful in the rest of the paper, as a special case of para-Markov chain with infinitely many states. For major details, see \cite{Facciaroni}.

Let us consider a para-Markov chain $ \ef := \{\ef_t,\ t\geq0\}$ as in Definition \ref{Definition of  para-Markov chains}. Let us assume a deterministic embedded chain 
\begin{align}
    Y_n = n, \qquad n\in\N
\end{align}
and $\lambda\in (0,\infty)$ constant. Then 
$\ef$ is said to be the \textit{Exchangeable fractional Poisson process} with parameters $\lambda$ and $\nu$. 

In this case the sequence of waiting times $\lg J_n^{(\nu)},\ n\geq1 \rg$    satisfies the following property: for each $n\geq1$,
\begin{align}
        P\l J_1^{(\nu)} > t_1,\ \ldots,\ J_n^{(\nu)} > t_n \r = 
 \mathcal{M}_{\nu}\left(-\lambda^{\nu} \left(\sum_{k=1}^n t_k\right)^{\nu}\right) \qquad \nu\in(0,1],\ \ \lambda \in (0,\infty).   \label{exchangeable fractional Poisson: joint survival function for (J_1, ..., J_n)}
\end{align}

From Equation \eqref{exchangeable fractional Poisson: joint survival function for (J_1, ..., J_n)} it follows that the distribution of $\ef$ reads

\begin{align}
    P\l \ef_t =n \r = \frac{(-t)^n}{n!} \mathcal{M}_\nu ^{(n)} (-\lambda ^\nu t^\nu) \qquad n\in\N .\label{Distribution of exchangeable fractional Poisson process}
\end{align}

An important property of the sequence of waiting times described in Equation \eqref{exchangeable fractional Poisson: joint survival function for (J_1, ..., J_n)} is that it forms an infinite \textit{Schur-constant} sequence. We recall that a sequence \( \{Z_k\}_{k=1}^\infty \) of non-negative random variables is said to be an infinite \textit{Schur-constant} sequence if, for any $n\geq1$, the joint survival function satisfies 
\begin{align*}
    P(Z_1 > t_1, Z_2 > t_2, \dots, Z_n > t_n) = S(t_1 + t_2 + \dots + t_n),
\end{align*}
for some function \( S \) that does not depend on \( n \).  
This structure represents a particular model of \textit{exchangeable} waiting times, as the function \( S \) depends on the \( t_k \) only through their sum. See \cite{Barlow,Spizzichino1,Spizzichino2} for details. This property justifies the name 
chosen for this counting process.

We note that for $\nu = 1$ we have that Equation \eqref{exchangeable fractional Poisson: joint survival function for (J_1, ..., J_n)} has the form
\begin{align}
        P\lq J_1 > t_1,\ \ldots,\ J_n > t_n \rq = e^{- \lambda \sum _{k=1}^n t_k},
        \end{align}
namely the waiting times $\lg J_n,\ n\geq1 \rg$ are i.i.d. exponential and $\mathcal{N}^1$ reduces to a Poisson process of parameter $\lambda$.

Moreover, we emphasize that from Equations \eqref{lamperti laplace} and \eqref{exchangeable fractional Poisson: joint survival function for (J_1, ..., J_n)}  one has
\begin{align}
    \l J_1^{(\nu)},\ldots, J_n^{(\nu)} \r \overset{d}{=} \frac{1}{\mathcal{L}}\l J_1,\ldots, J_n \r, \qquad n\geq1, \label{Preliminaries: relazione J_k esponenz e Mittag}
\end{align}

being $\lg J_n,\ n\geq1 \rg$ i.i.d. exponential of parameter $\lambda$. This also shows that $\lg J_n^{(\nu)},\ n\geq1\rg$ are exponential of random parameter $\lambda \Lam$. \\

We conclude this section by observing that $\lambda$ can be interpreted as a measure of the intensity of the process, in the sense that it controls the probability to have at least one event in a small time interval $\Delta t$. Indeed, using the relation of Theorem \ref{Proposition:  para-Makov chains as time-change of Markov chains through Lamperti}, we have
\begin{align*}
    P\l \ef _{t+\Delta t}- \ef _t=0 \r& = P\l \mathcal{N}^1 _{\mathcal{L}(t+\Delta t)}- \mathcal{N}^1 _{\mathcal{L}t}=0 \r \notag \\ &= \int _0^{\infty} e^{-\lambda y \Delta t} \, P(\mathcal{L} \in dy)\notag \\ &=  \mathcal{M} _\nu (-\lambda ^\nu (\Delta t)^\nu) \notag \\ &= 1-\frac{\lambda ^\nu (\Delta t)^\nu }{\Gamma (1+\nu)} + o \l \l \Delta t \r^{\nu} \r,
\end{align*}
which implies

\begin{align}
    P(\ef _{t+\Delta t}- \ef _t \geq 1) = \frac{\lambda ^\nu (\Delta t)^\nu }{\Gamma (1+\nu)} + o \l \l \Delta t \r^{\nu} \r. \label{comportamento locale}
\end{align}

We further observe that in the Poisson case $\lambda$ also has the meaning of mean number of events per unit time. Indeed, by the Poisson distribution, we have $\mathbb{E}\mathcal{N}^1 _t= \lambda$ if $t=1$. In contrast, the mean number of events per unit time in the exchangeable fractional Poisson process is infinite for $\nu\in(0,1)$ for each $\lambda$.

\section{Motion among random obstacles (Model 1)} \label{Sec: motion among random obstacles (model 1)}
We recall that Lorentz processes are models for the  motion of a particle among randomly distributed obstacles. 

In Gallavotti \cite{Gallavotti}, the set of obstacle centres is a Poisson point process; under the Boltzmann-Grad limit, the Lorentz process converges to the Markovian random flight defined in Section \ref{Markovian random flight}. As seen in Proposition \ref{limite diffusivo classico}, a further scaling limit leads to Brownian motion.

In this section, we shall define a Lorentz process based on a more general distribution of obstacles. To this aim, we define the  Mittag-Leffler point process, which includes the Poisson point process as a special case.

In this new setting, the obstacle system is more rarefied, making the model suitable for describing  more dilute gases. In particular, the free flight time of the particle has a distribution whose survival function has a power-law decay $Kt^{-\nu}$, $\nu \in (0,1),\ K > 0$, and thus has infinite expectation. This feature, in particular, has been previously ruled out by classical technical assumptions when studying Lorentz processes \cite{Spohn}. We shall prove that, under a Boltzmann-Grad type limit, the process converges to a non-Markovian random flight with infinite mean flight times. We will show that this random flight has a mean square displacement spreading, for $t \to +\infty$, as $ \widetilde{K}t^{2-\nu}$, $\widetilde{K}>0$ and furthermore it converges, under a second suitable limit, to an anomalous diffusive process.

For the sake of clarity, we believe it is appropriate to first describe the new random flight model in Section \ref{Subsection: Random flight Model 1}. Then, we shall define the Lorentz model in Section \ref{Sec: Mittag-Leffler distribution of obstacles} and finally we shall state the  limit theorems in Sections \ref{Sec: The Boltzmann-Grad limit modello 1} and \ref{Diffusion model 1}.

\subsection{A transport process with infinite mean flight times} \label{Subsection: Random flight Model 1}
We consider an isotropic transport process with infinite mean and stochastically dependent flight times, which is defined as in Definition \ref{Definizione isotropic transport process generica}, with the following assumptions.

\begin{assumption}{A1}{}\label{Isotropic: assumptions Model 1}
   Suppose that in Definition \ref{Definizione isotropic transport process generica} the joint distribution of the flight times $\lg J_n^{(\nu)},\ n\geq1 \rg$ is given, under all the probability measures $\mathds{P}^{(x,v)}$, by Equation \eqref{exchangeable fractional Poisson: joint survival function for (J_1, ..., J_n)} for some $\nu\in(0,1]$ and the (independent) speed is $\mathcal{C} = c$ almost surely, where $c$ is a positive constant.
\end{assumption}
In this case, we shall denote the random flight process of Definition \ref{Definizione isotropic transport process generica} by 
\begin{align*}
        \l {X}_{(1)}^\nu, {V}_{(1)}^\nu \r =\lg \l \Xl,  \Vl\r,\ t\geq0 \rg,  \qquad \nu \in (0,1].
\end{align*}
The assumption \ref{Isotropic: assumptions Model 1} implies that the number of direction changes up to time $t$ is  $\ef_t$, i.e. the exchangeable fractional Poisson process   defined in Section \ref{Exchengeable Fractional PP}.  

Let us consider initial data $(x,v)\in \R^d\times S^{d-1}$.  We will use the family of operators $\{\opmoduno_t\}_{t\geq0}$ acting on suitable functions and defined by
\begin{align}
    \opmoduno_t h (x,v)&:= \mathds{E}^{(x,v)}\lq h \l \Xl,  \Vl\r \rq \qquad h \in  C_0\l \R^d \times S^{d-1}\r , \label{uuu}
\end{align}
where $\mathds{E}^{(x,v)}$ denotes the expectation under the measure $\mathds{P}^{x,v}$.

We now obtain the counterpart of the evolution formula presented in Equation \eqref{Duhamel} (holding in the Markovian case) in the non-Markovian setting arising under \ref{Isotropic: assumptions Model 1}.

\begin{te}\label{Model 1: teorema su Duhamel}
   Let  $\l X^{(c,\lambda)}_t, V^{(c ,\lambda)}_t \r$ be the Markovian random flight defined in Section  \ref{Markovian random flight} and let $\mathcal{L}$ follow a Lamperti distribution as in Definition \ref{Definition of Lamperti distribution}. 
   Then
    \begin{align} \label{mixture Model 1}
        \l \Xl,  \Vl\r \overset{\text{d}}{=} \l X^{(c ,\lambda \mathcal{L})}_t, V^{(c ,\lambda \mathcal{L})}_t \r.
    \end{align}
    Moreover, the following formula holds
    \begin{align}
        \opmoduno_t h  = \, & \mathcal{M}_\nu (-\lambda ^\nu t^\nu)P_th  \notag \\ &+ \sum _{n=1}^\infty (-\lambda )^n \mathcal{M}_\nu ^{(n)}(-\lambda ^\nu t^\nu)  \notag  \int _{0< \tau _1< \tau _2< \dots < t}P_{t-\tau _n}L  P_{\tau _n-\tau _{n-1}}L \dots LP_{\tau _1}h\,  d\tau _1 \dots d\tau _n \notag
    \end{align}
    where $\mathcal{M}_\nu^{(n)}$ is defined in Equation \eqref{Derivative of Mittag Leffler function}.
\end{te}

\begin{proof}
Under the Assumptions \ref{Isotropic: assumptions Model 1},  the isotropic transport process in Definition \ref{Definizione isotropic transport process generica} reads, for each $t\geq0$,

\begin{align*}
    \l \begin{matrix}
        \Vl \\  \Xl
    \end{matrix} \r &= 
    \l \begin{matrix}
        v_m \\  
        x + \sum_{k = 1}^{m} cv_{k-1}J_k^{(\nu)} + c v_m\l t - \sum_{k = 1}^{m} J_k^{(\nu)} \r
    \end{matrix} \r
    \qquad
    \sum_{k = 1}^{m} J_k^{(\nu)} \leq t < \sum_{k = 1}^{m+1} J_k^{(\nu)} 
  \end{align*}  
where $m\geq 0$, which can compactly be re-written as
    \begin{align*}
 \l \begin{matrix}
        \Vl \\  \Xl
    \end{matrix} \r=    \sum_{m = 0}^{+\infty} \l \begin{matrix}
        v_m \\  
        x + \sum_{k = 1}^{m} cv_{k-1}J_k^{(\nu)} + c v_m\l t - \sum_{k = 1}^{m} J_k^{(\nu)} \r
    \end{matrix} \r 
    \mathds{1}_{\lg \sum_{k = 1}^{m} J_k^{(\nu)} \leq t < \sum_{k = 1}^{m+1} J_k^{(\nu)} \rg}.
\end{align*}

From Equation \eqref{Preliminaries: relazione J_k esponenz e Mittag} we have, for each $t\geq0$,
\begin{align*}
    \l \begin{matrix}
        \Vl \\  \Xl
    \end{matrix} \r \overset{d}{=} 
    \l \begin{matrix}
        v_m \\  
        x + \sum_{k = 1}^{m} cv_{k-1}\frac{J_k}{\mathcal{L}} + c v_m\l t - \sum_{k = 1}^{m} \frac{J_k}{\mathcal{L}} \r
    \end{matrix} \r
    \qquad
    \sum_{k = 1}^{m} \frac{J_k}{\mathcal{L}} \leq t < \sum_{k = 1}^{m+1} \frac{J_k}{\mathcal{L}}.
\end{align*}
where $\lg J_k / \Lam \rg$ are exponential random variables of random parameter $\lambda \Lam$. Therefore the equality in distribution of Equation \eqref{mixture Model 1} follows. 

Hence we can re-write Equation \eqref{uuu} as 
\begin{align*}
    \opmoduno_t h (x,v)&= \int _0^{\infty} \mathds{E}^{(x,v)} \lq h  \l X^{(c,\lambda l)}_{t}, V^{(c,\lambda l)}_{t} \r \rq \, \mathpzc{l}(dl)   \\
    & \notag  = \int _0^{\infty} \sgM^{(c, \lambda l)}_{t} h (x,v) \,  \mathpzc{l}(dl)
\end{align*}
where $\lg \sgM_t ^{(c,\lambda)},\ t\geq0 \rg$ is the Markovian semigroup defined in \eqref{semigruppo Boltzmann} and $\mathpzc{l}(\cdot)$ is the Lamperti measure as in Equation \eqref{Misura Lamperti distribution}. 
By using the representation of the Markovian operator in Equation \eqref{Duhamel} we obtain
\begin{align}
    \opmoduno_t h \, = \, &\int_0^{+\infty} e^{-\lambda l t}\, P_t\, h\, \mathpzc{l}(dl) \notag 
    \\ &+ \int_0^{+\infty} \l \sum_{n=1}^{+\infty} e^{-\lambda l t} (\lambda l)^n \int_{0 < \tau_1 < \cdots < \tau_n < t} P_{t-\tau_n} L \cdots L P_{\tau_1} \, h \, d\tau_1 \cdots d\tau_n \r \, \mathpzc{l}(dl).
    \label{here}
\end{align}

We now observe that we can interchange the summation and the integral in Equation \eqref{here}, by using the dominated convergence theorem for Bochner integrals (see \cite[Theorem 1.1.8]{abhn}). Consider
\begin{align*}
    f_k := \sum_{n=1}^{k} e^{-\lambda l t} (\lambda l)^n \int_{0 < \tau_1 < \cdots < \tau_n < t} P_{t-\tau_n} L \cdots L P_{\tau_1} h \, d\tau_1 \cdots d\tau_n .
\end{align*}

We have to show that $||f_k||$ is dominated by an integrable function, which is independent of $k$. Indeed, as $P_t$ and $L$ are contraction operators
\begin{align*}
    ||f_k|| &\leq \sum_{n=1}^{+\infty} \left\| e^{-\lambda l t} (\lambda l)^n \int_{0< \tau_1 < \cdots < \tau_n < t} P_{t-\tau_n} L \cdots LP_{\tau_1}\, h \, d\tau_1 \cdots d\tau_n \right\|\, \notag \\ 
     &\leq  \sum_{n=1}^{+\infty}  e^{-\lambda l t} (\lambda l)^n || h ||  \int_{0< \tau_1 < \cdots < \tau_n < t} d\tau_1 \cdots d\tau_n  \notag \\
     &= \sum_{n = 1}^{+\infty} e^{-\lambda l t} (\lambda l)^n ||h|| \frac{t^n}{n!}  = \, \left\| h\right\| \l 1 - e^{-\lambda l t}\r \notag.
\end{align*}
Moreover,
\begin{align*}
    \int_0^{+\infty} ||h|| \l 1 - e^{-\lambda l t}  \r \mathpzc{l}(dl) 
      = || h || ( 1 - \mathcal{M}_{\nu}\l-(\lambda t)^{\nu}\r.
\end{align*}

Then \cite[Theorem 1.1.8]{abhn} applies here and from Equation \eqref{here} we get
\begin{align}
    \opmoduno_t h \, = \, &\mathcal{M}_{\nu} (-\lambda^\nu t^\nu) P_th  \notag \\
    & + \sum_{n=1}^{+\infty} \int_0^{+\infty} \l e^{-\lambda l t} (\lambda l)^n \int_{0< \tau_1 < \cdots < \tau_n < t} P_{t-\tau_n} L \cdots LP_{\tau_1} h  \, d\tau_1 \cdots d\tau_n \r \mathpzc{l}(dl) \notag \\
    = \, &\mathcal{M}_{\nu} (-\lambda^\nu t^\nu) P_th \notag \\
    & + \sum_{n=1}^{+\infty} \l \int_0^{+\infty} e^{-\lambda l t} (\lambda l)^n \mathpzc{l}(dl) \r \l \int_{0< \tau_1 < \cdots < \tau_n < t} P_{t-\tau_n} L \cdots LP_{\tau_1} h  \, d\tau_1 \cdots d\tau_n \r  \notag\\
    = \,
    &\mathcal{M}_{\nu} (-\lambda^\nu t^\nu) P_th \notag \\
    & + \sum_{n=1}^{+\infty}  \l (-\lambda)^n \int_0^{+\infty} \l \frac{d}{d(\lambda t)}\r^n  e^{-\lambda l t} \mathpzc{l}(dl) \r \l \int_{0< \tau_1 < \cdots < \tau_n < t} P_{t-\tau_n} L \cdots LP_{\tau_1} h  \, d\tau_1 \cdots d\tau_n \r \notag \\
    \end{align}
Hence, by the dominated convergence theorem we can exchange the derivative and the integral in the first term of the product to obtain
    \begin{align}
   \opmoduno_t h
    = \, &\mathcal{M}_{\nu} (-\lambda^\nu t^\nu) P_th \notag \\
    & + \sum_{n=1}^{+\infty}    (-\lambda)^n  \mathcal{M}_{\nu}^{(n)}(-\lambda^\nu t^\nu) \int_{0< \tau_1 < \cdots < \tau_n < t} P_{t-\tau_n} L \cdots LP_{\tau_1} h  \, d\tau_1 \cdots d\tau_n,  \notag 
\end{align}
where we used the notation in Equation \eqref{Derivative of Mittag Leffler function}.
\end{proof}
In analogy with the Markovian case, it is possible to explicitly write the operator $\opmoduno_t$ as
\begin{align}
    & \opmoduno_t h(x,v) =  \mathcal{M}_\nu (-\lambda ^\nu t^\nu) h(x+cvt,v) + \notag \\
    &+ \sum _{n=1}^\infty  (-\lambda )^n \mathcal{M}_\nu ^{(n)}(-\lambda ^\nu t^\nu)   \int _{0< \tau _1< \tau _2< \dots < t} \int _{(S^{d-1})^{n}}\, h(x_t, \omega_{n})\, d\tau _1 \dots d\tau _n \, \mu (d\omega_1)\dots \mu (d\omega_n) \notag
\end{align}
where
\begin{align*}
    x_t:= x+ c\sum _{j=1}^{n} (\tau_j - \tau_{j-1}) \omega_{j-1} + c(t-\tau _{n})\, \omega_{n}, \qquad \omega_0 := v,
\end{align*}
and $v_t = \omega_n$ for $t\in [\tau_n, \tau_{n+1})$.

 We emphasize that the presented random flight is not trivially defined as a mixture of  Markovian ones. Indeed, the definition of the process depends on Assumptions \ref{Isotropic: assumptions Model 1}, and such a process turns out to be completely non-Markovian due to the depending structure of flight times. Second,  Theorem \ref{Model 1: teorema su Duhamel} just states an equality in distribution, which is proved as a consequence of the definition; this makes the random flight analytically tractable.

We now aim to find the mean squared  displacement of this non-Markovian random flight. We shall use the equality in distribution of Theorem \ref{Model 1: teorema su Duhamel}, which makes the transport process analytically tractable. 

\begin{prop} \label{Model 1: proposition on mean squared  displacement}
    Let $\lg \l \Xl,  \Vl\r,\ t\geq0 \rg$ be the non-Markovian random flight as defined by Assumptions \ref{Isotropic: assumptions Model 1}. Denote $\mathds{P}^{x,\mu} \coloneqq \int_{S^{d-1}} \mathds{P}^{x,v}\mu(dv)$ and by $\mathds{E}^{x,\mu}$ the corresponding expectation. Then we have
    \begin{align}
       \mathds{E}^{x,\mu} \left | \Xl - x \right |_e^2 = 2c^2 \int_0^t d w_1 \int_0^{w_1} dw_2\, \mathcal{M}_{\nu} \l -\lambda^{\nu} w_2^{\nu} \r, \label{Model 1: mean squared  displacement}
    \end{align}
    and then for large $t$ one has
    \begin{align}
        \mathds{E}^{x,\mu}\left | \Xl - x \right |_e^2 \sim \varepsilon\,t^{2-\nu}, \qquad \nu \in (0,1],\label{Model 1: mean squared  displacement asintotico}
    \end{align}
    where $\varepsilon$ is a positive constant.
\end{prop}
\begin{proof}
    Let us consider the mean squared  displacement of the Markovian random flight in Equation \eqref{Markovian: mean squared  displacement random flight}. Then by Theorem \ref{Model 1: teorema su Duhamel} we get
    \begin{align*}
        \mathds{E}^{x,\mu} \left | \Xl - x \right |_e^2  = 2c^2 \int_{0}^{\infty} \l \frac{t}{\lambda l}  - \frac{1}{\lambda^2 l^2} + \frac{e^{-\lambda t l}}{\lambda^2 l^2} \r  \mathpzc{l}(dl).
    \end{align*}
    Considering the Laplace transform in the time variable ($t \mapsto s$), we obtain, for any $s > 0$,
    \begin{align*}
         2c^2 \int_0^{\infty}\l \frac{1}{\lambda l s^2}  - \frac{1}{\lambda^2 l^2 s} + \frac{1}{\lambda^2 l^2 (s + \lambda l)} \r  \mathpzc{l}(dl) = 2c^2 \frac{1}{s^2} \int_0^{\infty} \frac{1}{s + \lambda l} \mathpzc{l}(dl) .
    \end{align*}
    By applying the Laplace transform to the second member of Equation \eqref{Model 1: mean squared  displacement} and using the definition of Lamperti distribution \eqref{lamperti laplace}, one has
    \begin{align*}
        2c^2 \frac{1}{s^2} \int_0^{\infty} e^{-sw_2} \mathcal{M}_{\nu} \l -\lambda^{\nu} w_2^{\nu} \r &= 2c^2 \frac{1}{s^2} \int_0^{\infty} e^{-sw_2} \int_0^{\infty} e^{-\lambda w_2 l} \mathpzc{l}(dl) \\
        &= 2c^2 \frac{1}{s^2} \int_0^{\infty} \frac{1}{s + \lambda l} \mathpzc{l} (dl),
    \end{align*}
    and the first part of the thesis follows. 

Now, denote $U(t) = \mathds{E}^{x,\mu} \left | \Xl - x \right |_e^2$ and note that
\begin{align}
    \int_0^{+\infty} e^{-s t} U'(t)dt = \, & 2c^2\frac{1}{s} \frac{s^{\nu-1}}{\lambda^\nu + s^\nu} \\ &\sim K's^{\nu-2} \qquad \text{as } s \to 0
\end{align}
for $K'>0$, where we used the well-known Laplace transform of the Mittag-Leffler survival function
\begin{align}
    \int_0^\infty e^{-st} \mathcal{M}_\nu (-\lambda^\nu t^\nu) dt \, = \, \frac{s^{\nu-1 }}{\lambda^\nu + s^\nu}.
\end{align}
It follows from the Karamata Tauberian Theorem \cite[Theorem XIII.5.2]{Feller} that
\begin{align}
      \mathds{E}^{x,\mu} \left | \Xl - x \right |_e^2 \sim K t^{2-\nu}
\end{align}
for $K>0$, as claimed.
\end{proof}

From Equation \eqref{Model 1: mean squared  displacement asintotico} we conclude that the random flight is super-diffusive for $\nu < 1$. Instead, for $\nu = 1$ we get the asymptotic behaviour of the Markovian random flight. Note that the asympotic behaviour is the same as observed in \cite{Barkai}, where the authors consider a one-dimensional transport process having power-law tailed flight times (even if  independent, contrary to what happens in our model).

\subsection{Mittag-Leffler distribution of obstacles} \label{Sec: Mittag-Leffler distribution of obstacles}
We preliminarly recall some basic notions on point processes.  We follow the notation of \cite{kallenberg}; another standard reference on point processes is \cite{Daley}.

Let $\mathcal{S} \in \mathcal{B} \l \mathbb{R}^3 \r$ and denote $\l \mathcal{S}, \mathcal{B} (\mathcal{S}), \mu \r$ a measure space which is finite on Borel subsets with finite Lebesgue measure. We denote by $\Pi$ a point process in $\mathcal{S}$, namely an integer-valued random measure. In particular, $\Pi$ is a locally finite kernel from a probability space $\l \Omega, \mathcal{A}, P \r$ into $\l \mathbb{R}^3, \mathcal{B} \l \mathbb{R}^3 \r \r$, i.e. a mapping $\Pi:(\Omega \times \mathcal{B}(\mathcal{S})) \mapsto \N$, such that $\Pi (\omega, B)$ is a locally finite integer-valued measure for fixed $\omega$ and an integer-valued random variable for fixed $B$.

From now on, let $\mu (B)= \rho |B|$, where $\rho \in (0,\infty)$ and $|B|$ denotes the Lebesgue measure of the Borel set $B$.

We recall that a point process is said to be a Poisson point process  with parameter $\rho \in (0, \infty) $, if, for any collection of mutually disjoint, finite  Borel sets $\{B_j,\ j \in\{1,\ldots, n\}\}$,  and any choice of non-negative integers $\{k_j,\ j \in\{1,\ldots, n\}\},\ n\geq 1$,  we have 
\begin{align}
    P\l \bigcap_{j = 1}^n \lg \mathcal{N} (B_j)= k_j \rg\r = e^{-\sum_{j=1}^n \rho |B_j|}\prod _{j=1}^n  \frac{\l \rho |B_j |\r ^{k_j}}{k_j !} \label{Inhomogeneous Poisson field as a special case of fractional poisson field}
\end{align}
where $\mathcal{N}(B)$ is the random variable $ \omega \mapsto \Pi (\omega, B)$, i.e. the random number of points inside $B$. Equation \eqref{Inhomogeneous Poisson field as a special case of fractional poisson field} indicates that $\mathcal{N}(B_j)$, $j=1, \dots, n$, are independent random variables, each having Poisson distribution with mean $\rho |B_j|$. 

We are ready to introduce a new point process, whose distribution is described, in addition to $\rho$, by a fractional parameter $\nu \in (0,1]$, such that  the Poisson case is re-obtained when $\nu =1$.   
We use the notation $\ef (B)$ to indicate the random number of points in $B\in\mathcal{B}(\mathcal{S})$.
Moreover, for the Mittag-Leffler function $\mathcal{M}_\nu(\cdot)$, we will use the notation in Equations \eqref{Mittag Leffler function} and \eqref{Derivative of Mittag Leffler function}.

\begin{defin} \label{Fractional Poisson field}
    A point process $\Pi$  is said to be a Mittag-Leffler point process 
    of parameters $\rho \in (0, \infty)$ and $\nu \in (0,1]$ if, for any collection of mutually disjoint, finite Borel sets $\{B_j,\ j \in\{1,\ldots, n\}\}$,  and any choice of non-negative integers $\{k_j,\ j \in\{1,\ldots, n\}\},\ n \geq 1$,  we have 
    \begin{align}
        & P\l \bigcap_{j = 1}^n \lg \ef (B_j)= k_j \rg \r 
        = (-1)^{k}\mathcal{M} _\nu^{ (k)  } \l -\l \sum _{j=1}^n \rho|B_j| \r^\nu \r\prod_{j = 1}^n \frac{(\rho |B_j|)^{k_j}}{k_j!}\ 
        \label{Fractional Poisson field formula}
    \end{align}
    where  $k := \sum_{j = 1}^n k_j$.
\end{defin}

In Appendix \ref{Appendix: Existence of Mittag-Leffler point process}, the existence and uniqueness of the point process of Definition \ref{Fractional Poisson field} are rigorously established. We further note that this point process has relationship with the one studied in \cite{Leonenkofield}.

At a first glance, from a purely formal perspective, the new point process is obtained by replacing the exponential function in Equation \eqref{Inhomogeneous Poisson field as a special case of fractional poisson field} with the Mittag-Leffler function in Equation \eqref{Fractional Poisson field formula}. This implies the loss of independence among the number of points in disjoint sets. 

Let us observe that for any finite Borel set $B \in \mathcal{B}(\mathcal{S})$, the random variable $\ef (B)$ has distribution
\begin{align} \label{un solo insieme}
        P(\ef (B)= k)= (-1)^k\mathcal{M} _\nu ^{(k)} (-  \rho^{\nu}|B|^\nu) \frac{(\rho|B|)^k}{k!} \qquad k\in\N , 
\end{align}
as a special case of Equation \eqref{Fractional Poisson field formula} for $n = 1$.

We further observe that, from a physical perspective, the Mittag-Leffler point process exhibits both similarities and differences compared to the Poisson case. Specifically, both are spatially homogeneous, meaning they are invariant under space translations. This follows from the fact that the distribution in Equation \eqref{Fractional Poisson field formula} only  depends  on the volumes of the $B_j$s.

On the other hand, the new point process is suitable to describe rarefied gases. To illustrate this point, we present the distribution of the random distance $ D_{\nu} $ between a given point $ y \in \R^3 $ and its nearest point of $\Pi$. Such a distribution does not depend on $y$ by space homogeneity. Note that $ y $ is not required to be a point of $\Pi$. 

It will become clear that, for $x\to \infty$, the survival function $P(D_{\nu}>x)$ decays to zero as a power law, which is a substantially different behaviour from the Poisson case.
Indeed, observe that $ D_{\nu} > x $ if and only if there are no points within the ball of radius $ x $ centered at $y$, denoted as $ B_x(y) $, for all $x\geq0$.
Hence, in the Poisson case, one has
\begin{align}
    P\l D_1 > x \r
    &= P \l \mathcal{N}\l B_{x}(y) \r = 0 \r = e^{-\rho \frac{4}{3}\pi x^3}, \notag 
\end{align}
while, in the Mittag-Leffler case, one has
\begin{align}
    P\l D_{\nu} > x \r &= P\l \ef\l B_{x}(y) \r = 0 \r = \mathcal{M}_{\nu}\l- \rho^{\nu} \l \frac{4}{3}\pi x^3 \r^{\nu} \r, \qquad \nu \in (0,1) . \label{Teoria cinetica: distribuzione del primo vicino in ML field}
\end{align}

Equation \eqref{Teoria cinetica: distribuzione del primo vicino in ML field} gives that $P\l D_{\nu} > x \r \sim 1/x^{3\nu}$ as $x\to\infty$.
This is consistent with the fact that the Mittag-Leffler model is  suited for situations where the gas exhibits a sparser structure, with large empty regions among obstacles.
It means that the moving particles in the gas are more likely to experience longer free paths or travel greater distances without interacting. In contrast, the Poisson model describes a gas where interactions or collisions occur more frequently, meaning the empty regions are smaller.

\begin{os} \label{Remark Intensity of parameter rho}
    In view of the main theorem \ref{Theorem BG boltzmann grad}, where the case of large $\rho$  will be considered, we observe that  $\rho$ controls the probability to have at least one point in a small region $B$. Indeed, we have   \begin{align}  
        P(\ef(B)\geq  1) & = 1-P(\ef(B)=0) \notag \\  & = 1-\mathcal{M}_\nu (-\rho ^\nu |B|^\nu) \notag \\     & = \frac{\rho ^\nu}{\Gamma (1+\nu)} |B|^\nu + o(|B|^\nu) \notag \qquad \nu \in (0,1].
    \end{align}
    We further observe that in the Poisson case, $\rho$ also has the meaning of expected number of points per unit volume; indeed, for the Poisson distribution, we have $\mathbb{E}\mathcal{N} (B)= \rho$ if $|B|=1$. In contrast, the expected number of points per unit volume in the Mittag-Leffler case is infinite  for each $\rho$ when $\nu\in(0,1)$.
\end{os}

Consider \(\Pi\) as in Definition \ref{Fractional Poisson field}. We now aim to determine the so-called Jánossy measure (see \cite[Section 4.3]{Penrose}) associated to $\Pi$.

\begin{lem} \label{Janossy measure for Mittaf-Leffler field}

Let us consider ${B}\in\mathcal{B}(\R^3),\ |B| <\infty$. Let $B$ contain $\ef (B)$ points of $\Pi$, which are located at the random positions $C_1, C_2, \dots, C_{\ef (B)}$. Then 
\begin{align*}
    P(C_1 \in dc_1, \dots, C_n\in dc_n, \ef (B)=n)=\frac{1}{n!}  \rho^n  \, (-1)^n \mathcal{M}_\nu^{(n)} \l -\rho^\nu |B|^{\nu} \r \, dc_1 \cdots dc_n,
\end{align*}
for each $n\geq1$.
\end{lem}

\begin{proof}
For each $n\geq1$, consider a sequence $\lg A_j,\ j\in\{1,\ldots,n\} \rg$ of mutually disjoint Borel sets s.t. $\cup_jA_j \subseteq {B}$. Then, using combinatorics arguments, we get
\begin{align*}
    P \l  C_1 \in A_1, \cdots C_{n} \in A_n , \ef (B) = n \r  &= \frac{1}{n!} P \l  \bigcap_{j = 1}^n\lg  \ef \l A_j \r = 1 \rg,\ \ef \l B\setminus \cup_{j=1}^n A_j \r = 0  \r \\
    &= \frac{1}{n!}\rho |A_1| \cdots \rho |A_n|  (-1)^n \mathcal{M}_\nu^{(n)} \l -  \rho^{\nu} |B|^\nu \r \\
    &= \frac{1}{n!} \rho^n   (-1)^n \mathcal{M}_\nu^{(n)} \l -\rho^\nu |B|^{\nu} \r  \int_{A_1}dc_1 \cdots \int_{A_n} dc_n \\
    &=  \frac{1}{n!}  \int_{A_1} \cdots \int_{A_n} \rho^n  \, (-1)^n \mathcal{M}_\nu^{(n)} \l -\rho^\nu |B|^{\nu} \r \, dc_1 \cdots dc_n,
\end{align*}
where in the second equality we used Equation  \eqref{Fractional Poisson field formula}.
Hence, the law of $(C_1,\ldots, C_n)$ is absolutely continuous and the thesis follows.

\end{proof}

\subsection{The Boltzmann-Grad limit} \label{Sec: The Boltzmann-Grad limit modello 1}
We now describe the new model of motion among random obstacles.
The obstacles are assumed to be identical hard spheres of radius $R$, whose centres are randomly distributed according to the Mittag-Leffler point process  \textcolor{black}{of parameters $\rho$ and $\nu$}, as in Definition \ref{Fractional Poisson field}. Some obstacles may overlap as it happens in the classical Gallavotti's model \cite{Gallavotti}.
Outside the system of obstacles, the particle moves along a straight line with constant speed $c>0$. When reaching the surface of an obstacle, an elastic collision occurs. This means that the particle is specularly reflected and the post-collisional speed is again equal to $c$. Each obstacle can be hit more than once, which leads to a long memory tail in the process. We now analyse the system in a Boltzmann-Grad type limit. 

\begin{os}
\textcolor{black}{The  Boltzmann-Grad limit consists in letting $R\to 0$ and simultaneously scaling the obstacle distribution so that, for any positive Borel set $B$, $\mathcal{N}_R ^{\nu} (B) \to \infty$, being  $\mathcal{N}_R ^{\nu} (B)$  the number of obstacle centres in  $B$. According to Spohn's results (see \cite{Spohn} and also Proposition 2.3 and Corollary 2.4 in \cite{Spohn_review}), the limiting random flight is Markovian if and only if the following equivalent conditions on the obstacle distribution are verified (we here translate Spohn's conditions in our language):}
\begin{enumerate}
\item \textcolor{black}{For any Borel set $B$, the scaled number of obstacles centers $R^2 \ef_R (B)$ converges in probability to a constant, as $R\to 0$.}

\item \textcolor{black}{For any collection of disjoint Borel sets $B_1, B_2, \dots, B_n$, the scaled variables $R^2 \ef_R (B_1)$, $R^2 \ef_R (B_2)$, \dots, $R^2 \ef_R (B_n)$ are independent in the limit $R\to 0$}.
\end{enumerate}
\textcolor{black}{Our Mittag-Leffler point process does not satisfy these conditions. Indeed, concerning point (1), the Laplace transform of $ R^2 \ef_R (B)$ is 
$$\mathbb{E} e^{-z R^2 \ef_R (B)} = \mathcal{M}_\nu \l  -\frac{\rho ^\nu}{R^{2\nu}} |B|^\nu \l 1-e^{-zR^2}\r^\nu \r \overset {R\to 0}{\longrightarrow} \mathcal{M}_\nu \l    -\rho ^\nu |B|^\nu z^\nu\r$$
where we assumed a scaling of the type $\rho \to \rho/R^2$, giving $\ef_R (B) \to \infty$.
Therefore we see that $ R^2 \ef_R (B)$ converges  to a  random variable, which reduces to a constant only in the case $\nu =1$. Concerning point $(2)$, we stress that our Mittag-Leffler point process is defined in such a way that the number of points in disjoint sets are dependent and this property is maintained in the limit $R \to 0$. Indeed, using the joint Laplace transform, we get
\begin{align}
\mathbb{E} e^{-z_1 R^2 \ef_R (B_1)\dots -z_n R^2 \ef_R (B_n)} = &\mathcal{M}_\nu \l  -\frac{\rho ^\nu}{R^{2\nu}} \l \sum _{j=1}^n |B_j| \l 1-e^{-z_jR^2}\r  \r^\nu \r \notag  \\
\overset {R\to 0}{\longrightarrow} &\mathcal{M}_\nu \l    -\rho ^\nu \l \sum _{j=1}^n   |B_j|z_j \r^\nu \r
\end{align}
which is factorisable only in the Poisson case corresponding to $\nu=1$}.
\end{os}

\textcolor{black}{In our notation, the above kinetic procedure can be easily expressed in the following way}:   let the obstacle radius $R$ approaches zero while the parameter $\rho$, representing the intensity of the process as described in Remark \ref{Remark Intensity of parameter rho}, diverges to infinity \textcolor{black}{(which is to say that $\mathcal{N} ^\nu(B) \to \infty$ for each Borel set $B$)}; these limits are taken in such a way that the product $R^2 \rho$ remains constant, that is

\begin{align} \label{BG limit}
    \rho \to \infty \qquad R \to 0 \qquad  \rho c \pi R^2 \to \lambda \in (0,\infty).
\end{align}

Without loss of generality, we shall put $\rho= \lambda / c \pi R^2$ and consider the one variable limit as $R\to 0$.

The constraint in Equation \eqref{BG limit} can be understood by studying the distribution of the free flight time. In Gallavotti's model, i.e. when $\nu =1$, for $R\to 0$ the distribution of the free flight time converges to an exponential distribution with mean $\lambda ^{-1}$. Hence, condition  \eqref{BG limit} means that the limit procedure is such that  the free flight time of the particle maintains a strictly positive, finite mean.

In this regard, in Lemma \ref{Distribuzione free flight time o tempo di volo} we shall see that this is not true if $\nu \in (0,1)$, given that  the free flight time has, in the present model, infinite mean both before and after the limit. 
Indeed, in this case the limit  \eqref{BG limit} simply ensures that the free flight time is almost surely positive and finite, meaning that it takes values in the open interval $(0,\infty)$ almost surely, and then it has a proper distribution.

So, suppose that a particle, which is initially located at $x\in \mathbb{R}^3$, is shot towards an arbitrary direction $v\in S^2$. Let   $T^R$  be the free flight time, i.e.  the first hitting time with the system of obstacles.
\begin{lem} \label{Distribuzione free flight time o tempo di volo}
The distribution of the free flight time $T^R$ has a discrete component in $0$, such that 
\begin{align*} 
P\l T^R =0 \r=1- \mathcal{M} _\nu \l -\rho ^\nu \l \frac{4}{3} \pi R^3 \r^\nu \r
\end{align*}
and an absolutely continuous component on  $(0,\infty)$, such that
\begin{align*} 
P\l T^R >t \r  = \mathcal{M} _\nu \l- \rho ^\nu \l\pi R^2 ct + \frac{4}{3} \pi R^3\r ^\nu \r \qquad t>0. 
\end{align*}
\end{lem}
\begin{proof}
    Let $t>0$. Then $T^R$ is greater than $t$ if and only if none of the obstacle centres lies in the set
    \begin{align}
        \theta (x,v,t) := \left\{y\in \mathbb{R}^3 : | y-(x+cvs)|_e \leq R,  \,\,\,  s\in [0,t] \right\}
    \end{align}
i.e. a cylinder of height $ct$ around the particle trajectory, capped by two hemispheres of radius $R$. So, for $t>0$, using formula \eqref{un solo insieme}, we have
\begin{align*} 
P\l T^R >t \r \notag &= P\l\ef (\theta (x,v,t)) =0 \r \\ &= \mathcal{M} _\nu \l- \rho ^\nu \l\pi R^2 ct + \frac{4}{3} \pi R^3\r ^\nu \r .
\end{align*}
The discrete mass at zero corresponds to the possibility that the starting point $x$ is located inside some obstacles:
\begin{align*} 
    P\l T^R =0 \r \notag &= P\l \ef (B_R(x))\geq 1 \r \\ \notag 
    &= 1-P \l  \ef (B_R(x)) =0 \r \\
    &=1- \mathcal{M} _\nu \l -\rho ^\nu \l \frac{4}{3} \pi R^3 \r^\nu \r,
\end{align*}
where $B_R(x)$ is the ball of radius $R$ centered at $x$.
\end{proof}

From Lemma \ref{Distribuzione free flight time o tempo di volo}, we observe that the distribution of $T^R$ does not depend either on the initial position $x$ or on the initial direction $v$, because the Mittag-Leffler random field is homogeneous and isotropic. Moreover, for large $t$ we have $P\l T^R > t\r \sim K t^{-\nu},\ K > 0$, hence the free flight time $T^R$ has infinite expectation.

\begin{os}
    From Lemma \ref{Distribuzione free flight time o tempo di volo}, we observe that  for $R\to 0$, the discrete component in $0$ vanishes and the distribution of $T^R$ converges to a Mittag-Leffler distribution, i.e.
    \begin{align}
        \lim _{R\to 0} P\l T^R>t \r=  \mathcal{M} _\nu \l-\lambda ^\nu t^\nu \r \qquad t\geq 0.
    \end{align}
    As mentioned before, after the limit \eqref{BG limit} the free flight time has a proper distribution with infinite expectation.
\end{os}

We are now ready to state the convergence theorem.

We denote by $ \mathcal{X}_t^{(1)}$ the position of the particle at time $t$ and by $ \mathcal{V}_t^{(1)}$ the unit velocity vector at time $t$, i.e. the direction of the particle at time $t$. We call $\l \X^{(1)}, \V^{(1)} \r := \lg \l \Xt, \Vt \r,\ t\geq0 \rg$ the Lorentz process in $\mathbb{R}^3 \times S^2$, with initial data given by the random variables $\l \X_0^{(1)}, \V_0^{(1)} \r$ in $\mathbb{R}^3 \times S^2$. Recall that the Lorentz process depends on the parameters of the Mittag-Leffler random field $\rho$ and $\nu$, and on the size of the obstacles $R$. However, to simplify the notation, we will not make this dependence explicit. On the other hand, the superscript indicates that it refers to the first model, while the second one will be discussed in the last section. For $\nu=1$, the obstacles have a Poisson distribution, and such Lorentz process coincides with Gallavotti's model. In the following theorem we shall use the notation 
\begin{align*}
    P^{(x,v)} \l \Xt\in dx', \Vt \in dv' \r := P \l \Xt\in dx', \Vt \in dv' \ | \ \mathcal{X}_0 = x, \mathcal{V}_0 = v \r.
\end{align*}

\begin{te} \label{Theorem BG boltzmann grad}
Consider the Lorentz process $\l \X^{(1)}, \V^{(1)} \r$ defined above.
Under the Boltzmann-Grad limit \eqref{BG limit}
the Lorentz process converges, in the sense of the single time distribution, to the random flight $\l {X}_{(1)}^\nu, {V}_{(1)}^\nu \r$ defined in Section \ref{Subsection: Random flight Model 1}, i.e.,
\begin{align*}
    P^{(x,v)} \l \Xt\in dx', \Vt \in dv' \r \to \mathds{P}^{(x,v)} \l X_{(1)}^\nu (t) \in dx', V_{(1)}^\nu (t) \in dv' \r
\end{align*}
weakly, for all $t>0$ and for all $(x,v) \in \mathbb{R}^3 \times S^2$.
\end{te}

\begin{proof}
Let us write $E^{(x,v)}$ for the expectation under $P^{(x,v)}$. Then for suitable functions $h(x,v)$, we can write
\begin{align*}
E^{(x,v)} h\l  \Xt,  \Vt\r  := \int _{S^2} \int _{B_{ct}(x)} h(x', v') P^{(x,v)} \l \Xt\in dx', \Vt \in dv' \r,
\end{align*}
where $B_{ct}(x)$ indicates the ball centered in $x$ of radius $ct$. Indeed, due to finite speed of the particle, the distribution of $\Xt$ is supported on $B_{ct}(x)$.

To get the statement, using the Portmanteau lemma, we can prove that
\begin{align}
    \lim _{R\to 0} E^{(x,v)}h\l \Xt, \Vt \r  = \mathds{E}^{x,v} h \l\Xl, \Vl\r  \qquad \forall \, \,\,h\in C_b(\mathbb{R}^3\times S^2),  \notag 
\end{align}
where $C_b(\mathbb{R}^3\times S^2)$ is the set of continuous bounded functions from $(\mathbb{R}^3\times S^2)$ to $\R$.

Among all possible paths, some allow each obstacle to be hit at most once, while others lead to recollisions. Therefore, we can partition the expectation as
\begin{align}
    E^{(x,v)} h\l \Xt, \Vt \r = E^{(x,v)}h\l \Xt, \Vt \r \mathds{1}_{\mathpzc{A}} + E^{(x,v)}h\l \Xt, \Vt\r\mathds{1}_{\mathpzc{A}^c},  \label{scomposizione della probabilita}
\end{align}
where $\mathpzc{A}$ is the event \textit{absence of recollisions}.

We first compute the term $E^{(x,v)} \mathds{1}_{\mathpzc{A}} h\l  \Xt,  \Vt \r$. 
To this aim, we use the decomposition
\begin{align*}
    \mathds{1}_{\mathpzc{A}} = \mathds{1}_{\mathpzc{A}} \l \mathds{1}_{\left\{T^R = 0\right\}} + \mathds{1}_{\left\{T^R > t\right\}} + \mathds{1}_{\left\{T^R \in(0,t)\right\}} \r,
\end{align*}
where $T^R$ is the first free flight time.
Using the distribution of $T^R$ in Lemma \ref{Distribuzione free flight time o tempo di volo} we have 
\begin{align*}
    \E\lq \mathds{1}_{\mathpzc{A}}\mathds{1}_{\left\{T^R = 0\right\}} h \l \Xt,  \Vt \r  \rq= h(x,v)\left[ 1- \mathcal{M}_\nu \l -\rho ^\nu \l \frac{4}{3} \pi R^3 \r^\nu \r \right]
\end{align*}
and
\begin{align*}
    \E \lq \mathds{1}_{\mathpzc{A}}\mathds{1}_{\left\{T^R > t\right\}}h\l  \Xt,  \Vt \r \rq  = \, h(x+cvt,v) \mathcal{M} _\nu \l - \rho ^\nu \l \pi R^2 ct + \frac{4}{3} \pi R^3 \r^\nu \r. 
\end{align*}

Finally, we consider the case of $T^R\in(0,t)$. 
In this time interval, the particle can hit the obstacles whose centres lie in the ball $B_{ct + R}\l\mathcal{X}_0\r$.
Let $M$ be the random number of obstacles inside $B_{ct + R}\l \mathcal{X}_0 \r$ and let $C_i,\ i=1,\ldots, M$ be the random positions of their centres. In compact notation, $\underline{C}:= (C_1, \dots, C_M)$. Lowercase letters shall indicate the realizations of such variables.

From now on, we shall write $ \Xt =  \Xt(\underline{C})$ and $ \Vt =  \Vt(\underline{C})$ to make the fact explicit that position and velocity can be written as a function of the location of the obstacles. Moreover, we shall indicate with $\mathcal{H}_m$ the subset of $\l B_{ct + R}\l \mathcal{X}_0^{(1)} \r \r^m$ such that, conditionally on $(x,v)$, the particle at least hits $1$ obstacle up to time $t$ and has no recollisions. By using Lemma \ref{Janossy measure for Mittaf-Leffler field} we can write
\begin{align}
    & E^{(x,v)}\lq \mathds{1}_{\mathpzc{A}}\mathds{1}_{\left\{T^R \in (0,t)\right\}} h\l  \Xt\l \underline{C} \r,  \Vt \l \underline{C} \r \r \rq \nonumber\\
    &=\sum_{m = 1}^{\infty} \int_{\mathcal{H}_m} h\l  \Xt(\underline{c}),  \Vt(\underline{c})\r  P^{(x,v)}  \l  C_1 \in dc_1, \cdots, C_{m} \in dc_m , \ef \l B_{ct + R}\l \mathcal{X}_0^{(1)} \r \r = m \r \nonumber\\
    &= \sum_{m = 1}^{\infty} \int_{\mathcal{H}_m} h\l  \Xt(\underline{c}),  \Vt(\underline{c})\r  \frac{\rho^m (-1)^m}{m!} \mathcal{M}_\nu^{(m)}\l -\rho ^\nu |B_{ct + R}\l \mathcal{X}_0^{(1)} \r | ^\nu \r \, dc_1dc_2\dots dc_m . \label{cuore}
\end{align}

Let $n$ be the number of hit obstacles out of $m$ and let $\mathcal{F}^{(n)}$ be the subset of $\mathcal{H}_m$ containing all the obstacle configurations such that the particle exactly hits  $n$ obstacles\footnote{Notably, collisions must occur before time $t$, ensuring that the hit obstacles are not too far apart; otherwise, the particle would not have sufficient time to reach the first obstacle and then adjust its trajectory to hit the second.
}. Hence 
$$\mathcal{H}_m= \bigcup _{n=1}^m \mathcal{F}^{(n)}   $$
and the above integral reads
$$
\sum_{n = 1}^{\infty} \sum_{m = n}^{\infty}
    \int_{\mathcal{H}_m} \delta _{\mathcal{F}^{(n)}}(\underline{c}) \, h\l  \Xt(\underline{c}),  \Vt(\underline{c})\r \frac{\rho^m (-1)^m}{m!} \mathcal{M}_\nu^{(m)}\l -\rho ^\nu \left| B_{ct + R}\l \mathcal{X}_0^{(1)} \r \right| ^\nu \r \, dc_1dc_2\dots dc_m, 
    $$
where $\delta _{\mathcal{F}^{(n)}}(\underline{c}) = 1$ if $\underline{c} \in \mathcal{F}^{(n)}$ and 0 otherwise, and we exchanged the order of summations.

Now, for each $n$, there are $\binom{m}{n}$ possible ways to choose which obstacles are actually hit, i.e.
$$ \mathcal{F} ^{(n)}= \bigcup _{j=1}^{\binom{m}{n}} \mathcal{F}^{(n)}_j $$ 
where $\mathcal{F}^{(n)}_j$ is the $j$-th subset of $\mathcal{F}^{(n)}$ of all the possible configurations with $n$ hit obstacles; hence $\delta _{\mathcal{F}^{(n)}_j} \delta _{\mathcal{F}^{(n)}} = \delta _{\mathcal{F}^{(n)}_j}$. Thus the above integral becomes
$$
\sum_{n = 1}^{\infty} \sum_{m = n}^{\infty}
    \sum _{j=1}^{\binom{m}{n}}\int_{\mathcal{H}_m} \delta _{\mathcal{F}^{(n)}_j}(\underline{c}) \, h\l  \Xt(\underline{c}),  \Vt(\underline{c})\r \frac{\rho^m (-1)^m}{m!} \mathcal{M}_\nu^{(m)}\l -\rho ^\nu \left| B_{ct + R}\l \mathcal{X}_0^{(1)} \r \right| ^\nu \r \, dc_1dc_2\dots dc_m.
    $$

Now, given that marginally each center is uniformly distributed on the sphere, 
all the possible $\binom{m}{n}$ choices give the same contribution. Hence, denoting by $\mathcal{F}^{(n)}_1$ the configuration such that the particle hits the first $n$ obstacles of the sequence, we have 
    $$ \sum_{n = 1}^{\infty} \sum_{m = n}^{\infty}
    \binom{m}{n} \int_{\mathcal{H}_m} \delta _{\mathcal{F}^{(n)}_1}(\underline{c}) \, h\l  \Xt(\underline{c}),  \Vt(\underline{c})\r \frac{\rho^m (-1)^m}{m!} \mathcal{M}_\nu^{(m)}\l -\rho ^\nu \left| B_{ct + R}\l \mathcal{X}_0^{(1)} \r \right| ^\nu \r \, dc_1dc_2\dots dc_m.
    $$
 Now, let 
 $\underline{C}  = (\underline{\widetilde{C}}, \underline{\hat{C}})$, with $\underline{\widetilde{C}} = (C_1,\ldots, C_n)$ and  $\underline{\hat{C}} = (C_{n+1},\ldots, C_m)$. We remark that actually in the above integral, the position $\mathcal{X}_t^{(1)}$ and the velocity $\mathcal{V}_t^{(1)}$ of the particle only depend on the hit obstacles $\underline{\widetilde{C}}$. 

We indicate the tube-like flow induced by the hit obstacles $\underline{\widetilde{c}}$ 
\begin{align*}
    \Theta\l\underline{\widetilde{c}}\r := \left\{ y\in B_{ct + R}\l \mathcal{X}_0^{(1)} \r  \ \text{s.t.}\ \left| \mathcal{X}^{(1)}_s(\underline{\widetilde{c}}) - y \right|_e \leq R,\ s\in[0,t] \right\}
\end{align*}
while its complement in $B_{ct + R}\l \mathcal{X}_0^{(1)} \r $ is denoted by $\overline{\Theta\l\underline{\widetilde{c}}\r}$. We use $\mathcal{K}^{(n)}$ for the set of feasible configurations of $\widetilde{\underline{c}}$ and thus we can split the integral on $\mathcal{F}_1^{(n)}$ as follows

\begin{align}
    \sum_{n = 1}^{\infty} \sum_{m = n}^{\infty} \binom{m}{n}&\int_{\mathcal{K}^{(n)} } dc_1 \ldots dc_n  h\l  \Xt(\underline{\widetilde{c}}),  \Vt(\underline{\widetilde{c}})\r
    \notag \\
    &\int_{\l \overline{\Theta\l\underline{\widetilde{c}}\r}\r ^ {m-n}} dc_{n+1}\ldots dc_m \frac{\rho^m (-1)^m}{m!} \mathcal{M}_\nu^{(m)}\l -\rho ^\nu \left| B_{ct + R}\l \mathcal{X}_0^{(1)} \r \right| ^\nu \r . 
    \notag 
\end{align}

Using Equation \eqref{lamperti laplace} in the definition of the Lamperti random variable, we note that 

\begin{align*}
    \frac{\rho^m (-1)^m}{m!} \mathcal{M}_\nu^{(m)}\l -\rho ^\nu \left| B_{ct + R}\l \mathcal{X}_0^{(1)} \r \right| ^\nu \r  &= \frac{\rho^m (-1)^m}{m!}\int_0^{\infty} (-1)^m l^m  e^{-\rho \left| B_{ct + R}\l \mathcal{X}_0^{(1)} \r \right| l} \ P\l \Lam \in dl \r \\
    &=  \int_0^{\infty} \frac{(\rho l)^m}{m!}  e^{-\rho\left| B_{ct + R}\l \mathcal{X}_0^{(1)} \r \right| l} \ P\l \Lam \in dl \r. 
\end{align*}
Taking into account that
\begin{align*}
    \sum_{m = n}^{\infty} \frac{(\rho l)^{m-n}}{(m-n)!} \int_{\l \overline{\Theta\l\underline{\widetilde{c}}\r} \r^{m-n}}dc_{n+1}\ldots dc_{m} 
    &= e^{\rho l \left|\overline{\Theta\l\underline{\widetilde{c}}\r}\right|},
\end{align*}
and $\left| \Theta \l \underline{\widetilde{c}}\r  \right| = \left| B_{ct + R}\l \mathcal{X}_0^{(1)} \r  \right| - \left|\overline{\Theta\l\underline{\widetilde{c}}\r}\right|$,
Equation \eqref{cuore} becomes
\begin{align}
    &E^{(x,v)} \lq \mathds{1}_{\mathpzc{A}}\mathds{1}_{\left\{T^{R} \in (0,t)\right\}} h\l  \Xt\l \underline{\widetilde{C}} \r,  \Vt \l \underline{\widetilde{C}} \r \r \rq \nonumber \\
    &= \sum_{n = 1}^{\infty} \int_{\mathcal{K}^{(n)} }dc_1\ldots dc_n \int_0^{\infty} P^{(x,v)} (\mathcal{L}\in dl) \frac{(\rho l)^n}{n!} h\l  \Xt(\underline{\widetilde{c}}),  \Vt(\underline{\widetilde{c}})\r e^{-\rho l \left| B_{ct + R}\l \mathcal{X}_0^{(1)} \r \right| } \nonumber \\
    & \sum_{m = n}^{\infty} \frac{(\rho l)^{m-n}}{(m-n)!}\int_{\l \overline{\Theta \l \underline{\widetilde{c}}\r } \r^{m-n}} dc_{n+1}\ldots dc_m \nonumber \\
    &= \sum_{n = 1}^{\infty} \int_{\mathcal{K}^{(n)} }dc_1\ldots dc_n \int_0^{\infty} P^{(x,v)} (\mathcal{L}\in dl) \frac{(\rho l)^n}{n!} h\l  \Xt(\underline{\widetilde{c}}),  \Vt(\underline{\widetilde{c}})\r e^{-\rho l \left| B_{ct + R}\l \mathcal{X}_0^{(1)} \r \right| } e^{\rho l \left| \overline{\Theta \l \underline{\widetilde{c}}\r } \right|}  \nonumber \\
    &= \sum_{n = 1}^{\infty} \int_{\mathcal{K}^{(n)} }dc_1\ldots dc_n \int_0^{\infty} P^{(x,v)} (\mathcal{L}\in dl) \frac{(\rho l)^n}{n!} h\l  \Xt(\underline{\widetilde{c}}),  \Vt(\underline{\widetilde{c}})\r e^{-\rho l \left| \Theta \l \underline{\widetilde{c}}\r  \right|}. \nonumber
\end{align}

Moreover, using  Equation \eqref{lamperti laplace} again for the Laplace transform of the Lamperti distribution, the above equation can be written as

\begin{align}
    & \sum_{n = 1}^{\infty} \int_{\mathcal{K}^{(n)} }dc_1\ldots dc_n \frac{(-\rho)^n}{n!} h\l  \Xt(\underline{\widetilde{c}}),  \Vt(\underline{\widetilde{c}})\r \int_0^{\infty}P(\mathcal{L}\in dl) (-l)^n e^{-\rho l \left| \Theta \l \underline{\widetilde{c}}\r  \right|} \nonumber \\
    &=  \sum_{n = 1}^{\infty} \int_{\mathcal{K}^{(n)} }dc_1\ldots dc_n \frac{(-\rho)^n}{n!} h\l  \Xt(\underline{\widetilde{c}}),  \Vt(\underline{\widetilde{c}})\r \lq \l\frac{d}{dz}\r^n \mathcal{M}_{\nu}(-z^{\nu}) \rq \bigg|_{z = \rho \left| \Theta\l \underline{\widetilde{c}} \r \right|} \notag.
\end{align}
We now observe that the $n$ obstacles can be hit in $ n! $ different chronological orders. Let $ \mathcal{K}^{(n)}_i $ denote the $ i $-th such ordering, i.e.  
\begin{align*}
    \mathcal{K}^{(n)} = \bigcup_{i = 1}^{n!}\mathcal{K}^{(n)}_i 
\end{align*}
and let $\mathcal{K}^{(n)}_1$ be the configuration in which the obstacle centered at $c_j$ is the $j$-th one to be hit in chronological order.
By using the notation for the derivative of the Mittag-Leffler function in Equation \eqref{Derivative of Mittag Leffler function}, Equation \eqref{cuore} can be written as
\begin{align}
    & E^{(x,v)} \lq         \mathds{1}_{\mathpzc{A}}\mathds{1}_{\left\{T^{R} \in (0,t)\right\}} h\l  \Xt\l \underline{\widetilde{C}} \r,  \Vt \l \underline{\widetilde{C}} \r \r \rq  \nonumber \\
    &=\sum_{n = 1}^{\infty} \int_{\mathcal{K}^{(n)} }dc_1\ldots dc_n \frac{(-\rho)^n}{n!} h\l  \Xt(\underline{\widetilde{c}}),  \Vt(\underline{\widetilde{c}})\r \mathcal{M}^{(n)}_{\nu}\l-\rho^{\nu} \left| \Theta\l \underline{\widetilde{c}} \r \right|^{\nu}\r  \nonumber \\
    &= \sum_{n = 1}^{\infty} \sum_{i = 1}^{n!}\int_{\mathcal{K}^{(n)}_i }dc_1\ldots dc_n \frac{(-\rho)^n}{n!} h\l  \Xt(\underline{\widetilde{c}}),  \Vt(\underline{\widetilde{c}})\r \mathcal{M}^{(n)}_{\nu}\l-\rho^{\nu} \left| \Theta\l \underline{\widetilde{c}} \r \right|^{\nu}\r  \nonumber \\
    &= \sum_{n = 1}^{\infty} \int_{\mathcal{K}^{(n)} _1 }dc_1\ldots dc_n (-\rho)^n h\l  \Xt(\underline{\widetilde{c}}),  \Vt(\underline{\widetilde{c}})\r \mathcal{M}^{(n)}_{\nu}\l-\rho^{\nu} \left| \Theta\l \underline{\widetilde{c}} \r \right|^{\nu}\r . \label{Theorem: BG limit before change variable without n! senza n fattoriale}
\end{align}
We now perform a change of variables in the above integral, following \cite{Gallavotti}. Let $\tau _1, \dots, \tau _n$ be the collision times, such that $0< \tau _1 < \tau _2 < \dots \tau _n < t$ and let $v_1, v_2, \dots, v_n$ be the unit velocity vectors that emerge in the sequence of collisions. 

Since no recollisions occur here, there is a one to one correspondence between $\underline{\widetilde{c}}$ and $\underline{\widetilde{\tau}} := (\tau _1, \dots, \tau _n)$,\ $\underline{\widetilde{v}} := ( v_1, \dots, v_n)$. We thus perform a change of variables in the integral in Equation \eqref{Theorem: BG limit before change variable without n! senza n fattoriale}. The Jacobian of the transformation leading to the new variables  has the form $\frac{c^n}{4^n} R^{2n}$. It means that Equation \eqref{Theorem: BG limit before change variable without n! senza n fattoriale} reads
\begin{align*}
    &\sum_{n = 1}^{\infty} \int_{\mathcal{K}^{(n)}_1 }dc_1\ldots dc_n (-\rho)^n h\l  \Xt(\underline{\widetilde{c}}),  \Vt(\underline{\widetilde{c}})\r \mathcal{M}^{(n)}_{\nu}\l-\rho^{\nu} \left| \Theta\l \underline{\widetilde{c}} \r \right|^{\nu}\r  \\
    &= \sum_{n = 1}^{\infty} \int _{0 < \tau _1 < \tau _2 \leq \dots < \tau_n < t} \int _{(S^2)^{n}}(-\rho)^n h\l  \Xt(\underline{\widetilde{\tau}}, \underline{\widetilde{v}}),  \Vt(\underline{\widetilde{\tau}}, \underline{\widetilde{v}})\r \mathcal{M}^{(n)}_{\nu}\l- \rho^{\nu} \left| \Theta\l \underline{\widetilde{\tau}}, \underline{\widetilde{v}} \r \right|^{\nu}\r \\
    &\qquad\qquad\qquad\qquad\qquad\qquad \frac{c^n}{4^n} R^{2n} dv_1\dots dv_n\,  d\tau _1 \dots d\tau _n \\
    &= \sum_{n = 1}^{\infty} \int _{0< \tau _1 < \tau _2 < \dots < \tau_n < t} \int _{(S^2)^{n}}(-\rho)^n h\l  \Xt(\underline{\widetilde{\tau}}, \underline{\widetilde{v}}),  \Vt(\underline{\widetilde{\tau}}, \underline{\widetilde{v}})\r \mathcal{M}^{(n)}_{\nu}\l- \rho^{\nu} \left| \Theta\l \underline{\widetilde{\tau}}, \underline{\widetilde{v}} \r \right|^{\nu}\r \\
    &\qquad\qquad\qquad\qquad\qquad\qquad c^n \pi^n R^{2n} \mu(dv_1)\dots \mu(dv_n)\,  d\tau _1 \dots d\tau _n 
\end{align*}
where in the last step we denoted by $\mu(dv_i)= dv_i/4\pi$ the uniform probability measure on the unit sphere.

Putting all together, the first summand of Equation \eqref{scomposizione della probabilita} becomes
\begin{align} \label{densita processo di Lorentz}
    &  E^{(x,v)}h\l \Xt, \Vt \r \mathds{1}_{\mathpzc{A}} =  h(x,v)\left[ 1- \mathcal{M}_\nu \l -\rho ^\nu \l \frac{4}{3} \pi R^3 \r^\nu \r \right] + \notag \\
    & \notag + h(x+cvt,v) \mathcal{M} _\nu \l - \rho ^\nu \l \pi R^2 ct + \frac{4}{3} \pi R^3 \r^\nu \r + \\
    & \notag + \sum_{n = 1}^{\infty} \int _{0< \tau _1 < \tau _2 < \dots < \tau_n < t} \int _{(S^2)^{n}}(-\rho)^n h\l  \Xt(\underline{\widetilde{\tau}}, \underline{\widetilde{v}}),  \Vt(\underline{\widetilde{\tau}}, \underline{\widetilde{v}})\r \mathcal{M}^{(n)}_{\nu}\l- \rho^{\nu} \left| \Theta\l \underline{\widetilde{\tau}}, \underline{\widetilde{v}} \r \right|^{\nu}\r \\
    &\qquad\qquad\qquad\qquad\qquad\qquad c^n \pi^n R^{2n} \mu(dv_1)\dots \mu(dv_n)\,  d\tau _1 \dots d\tau _n .
\end{align}
Passing to the Boltzman-Grad limit \eqref{BG limit} in Equation \eqref{densita processo di Lorentz}, we obtain that
\begin{align}
    \lim _{R\to 0}E^{(x,v)}h\l \Xt, \Vt \r \mathds{1}_{\mathpzc{A}} &=  h(x+cvt,v) \, \mathcal{M} _\nu \l- \lambda ^\nu t^\nu \r + \sum _{n=1}^\infty (-1)^n \mathcal{M}_\nu ^{(n)} (-\lambda ^\nu t^\nu)\,  \lambda ^n\notag   \\
     & \int _{0< \tau _1 < \tau _2 < \dots < \tau_n < t} \int _{(S^2)^{n}}h\l x_t, v_n \r
   \mu(dv_1)\dots \mu (dv_n) d\tau _1 \dots d\tau _n, \label{Theorem BG model 1: fine parte senza ricollisioni e riconoscimento duhamel}
\end{align}
where
\begin{align*}
    x_t:= x+ c\sum _{j=1}^{n} (\tau_j - \tau_{j-1}) v_{j-1} + c(t-\tau _{n})\, v_{n}, \qquad v_0 := v. 
\end{align*}
and we have taken into account that the unit velocity vector at time $t$ is $v_n$.
For details on sufficient conditions to exchange the limit for $R\to 0$ and the sum over $n$ in Equation \eqref{densita processo di Lorentz} see Appendix \ref{Appendix: appendice sulla convergenza dominata}.

The right side of \eqref{Theorem BG model 1: fine parte senza ricollisioni e riconoscimento duhamel}    coincides with the Duhamel expansion of $ \mathds{E}^{x,v} h \l\Xl, \Vl\r$  of the random flight defined in Section \ref{Subsection: Random flight Model 1}, as shown in Theorem \ref{Model 1: teorema su Duhamel} in the special case $d=3$.

To conclude this proof, it is sufficient to verify that the recollision term $E^{(x,v)}h\l \Xt, \Vt\r\mathds{1}_{\mathpzc{A}^c}$ 
in Equation \eqref{scomposizione della probabilita} vanishes for $R\to 0$. Firstly, let us rewrite Equation \eqref{Theorem BG model 1: fine parte senza ricollisioni e riconoscimento duhamel} explicitly as  
\begin{align}
    \lim _{R\to 0} & \int _{S^2} \int _{B_{ct}(x)} h(x',v') P^{(x,v)}\l \Xt \in dx',  \Vt\in dv', \mathpzc{A}\r \notag\\
    &= \int _{S^2} \int _{B_{ct}(x)} h(x', v') P^{(x,v)} \l \Xl\in dx', \Vl \in dv' \r. \label{fff}
\end{align}
The above considerations are true for each  function $h\in C_b \l \mathbb{R}^3\times S^2 \r$, and in particular for a function $h$ which is equal to $1$ on the domain $B_{ct}(x)\times S^2$, in which case Equation \eqref{fff} reads
\begin{align} \notag
    \lim _{R\to 0} & \int _{S^2} \int _{B_{ct}(x)}  P^{(x,v)}\l  \Xt \in dx',  \Vt\in dv', \mathpzc{A} \r = 1, 
\end{align}
i.e.
\begin{align*}
    \lim_{R\to0} P^{(x,v)} \l \mathpzc{A} \r = 1 \qquad \text{and then} \qquad \lim_{R\to0} P^{(x,v)} \l \mathpzc{A}^c \r = 0.
\end{align*}
Finally, we observe that 
\begin{align*}
     \lim _{R\to 0} &\left| E^{(x,v)}h\l \Xt, \Vt \r \mathds{1}_{\mathpzc{A}^c}\right| \\ \notag 
    &\leq   \underset {{B_{ct}(x) \times S^2}}{\sup} |h| \, \lim _{R\to 0}  \int _{S^2} \int _{B_{ct}(x)}  P^{(x,v)}\l \Xt \in dx',  \Vt\in dv', \mathpzc{A}^c \r  \\
    & = \underset {{B_{ct}(x) \times S^2}}{\sup} |h| \, \lim _{R\to 0}P^{(x,v)} \l \mathpzc{A}^c \r    = 0,
\end{align*}
and this concludes the proof.
\end{proof}

\begin{os}
The above theorem shows that, under the Boltzmann-Grad limit \eqref{BG limit} we obtain  a resulting random flight which is easier to be handled with respect to the Lorentz process, because, as we have seen in the proof, the measure of paths having recollisions tends to zero. 
However, there is a crucial difference between the present model and Gallavotti's one.
Specifically, in  Gallavotti's model, the elimination of recollisions leads to Markovianity. Indeed,  the resulting random flight is the Markovian transport process recalled in Section \ref{Markovian random flight}, which is in turn approximated by a diffusion process by a further scaling limit. 
Instead, when $\nu \in (0,1)$, even if the probability of recollisions vanishes, the resulting random flight  is not Markovian. This is due to the Mittag-Leffler point process that generates infinite mean, power law, free flight times of the particle: these flight times remain with infinite expectation even in the limit. Moreover, we shall see in the next section that, under a proper scaling, it leads to an anomalous diffusive behaviour.
\end{os}

\subsection{Convergence to anomalous diffusion} \label{Diffusion model 1}

We now aim to study the asymptotic behaviour of the random flight defined in Section \ref{Subsection: Random flight Model 1}. Specifically, we show that, under an appropriate limiting regime, the random flight weakly converges to an anomalous diffusive process. In the following, we define such process, by means of its finite dimensional distributions. We shall indicate with $\langle\cdot,\cdot\rangle$ the usual inner product.

\begin{defin} \label{Definition: anomalous diffusion}
    Consider a stochastic process $W = \lg W_t,\ t\geq 0 \rg$ on $\mathbb{R}^d$, $d\geq 1$. We say that it is a Mittag-Leffler anomalous-diffusion process of parameter $\nu\in(0,1]$ if, for any choice of times $0 \leq t_1 <\ldots < t_n$,  the vector $\Gamma = \l W_{t_1}, \ldots, W_{t_n} \r$ has the characteristic function
    \begin{align}
        \varphi_{\nu}^n(u) = \mathcal{M}_{\nu} \l - \l \frac{1}{2} \langle u, Q u \rangle\r^{\nu} \r \qquad \forall u\in \R^{nd},\ n\geq1, \label{funzione caratteristica finito dimensionali}
    \end{align}
    where the matrix $Q\in \R^{nd\times nd}$ is a block diagonal matrix,  and the $j$th block has the form
    \begin{align}
        Q_j =  [t_h \wedge  t_k]_{h,k \in \{1,\ldots, n\}} \qquad j= 1,\ldots, d. \label{Anomalous diff 1: forma del j blocco nella definizione}
    \end{align}
\end{defin}

We observe that for $\nu=1$ the above process reduces to a standard Brownian motion on $\mathbb{R}^d$, with var-cov matrix $Q$. Indeed, $\mathcal{M}_{1}(x)=e^x$ and thus the characteristic function in Equation \eqref{funzione caratteristica finito dimensionali} becomes that of Brownian motion. See \cite[page 23]{sato} for details.

From Equation \eqref{funzione caratteristica finito dimensionali}, we see that $ \lg W_t,\ t\geq 0 \rg$ is self-similar with Hurst index equal to $1/2$. We indeed recall that a process $\lg Y_t,\ t\geq 0 \rg$ is said to be self-similar with Hurst index $H>0$ if, for any $a>0$, we have that $\lg Y_{at},\ t\geq 0 \rg$ has the same finite dimensional distributions of  $\lg a^H Y_t,\ t\geq 0 \rg$.
Furthermore, the process $W$ turns out to be a particular case of the randomly scaled Gaussian processes studied in \cite[Sect. 3]{Pagnini}, as we shall see later.

Putting $n=1$ in Equation \eqref{funzione caratteristica finito dimensionali}, we have that $W_t$ has characteristic function
    \begin{align} 
        \varphi_\nu^1 (u) = \mathcal{M}_{\nu} \l -  \frac{1}{2^\nu} |u|_e^{2\nu} t^\nu       \r \qquad \forall u\in \R^{d} \label{funzione caratteristica diffusione anomala}
    \end{align}
from which it can be seen that each component of $W_t$ has infinite variance. Hence this process is characterized as a super-diffusion in physical terms.

\begin{os} \label{Anomalous diffusion: remark equazione governante}
We observe that $W_t$ has a density $q(x,t)$, which is the fundamental solution to the fractional equation
\begin{align} 
    \frac{\partial ^\nu}{\partial t ^\nu} q(x,t)= - \frac{1}{2^\nu} (-\Delta )^\nu q(x,t).
\label{fractional diffusion equation}
\end{align}
In Equation \eqref{fractional diffusion equation}, the operator $\frac{\partial ^\nu}{\partial t ^\nu}$ denotes the Caputo fractional derivative, while  $-(-\Delta )^\nu$ denotes the fractional Laplacian. To prove \eqref{fractional diffusion equation} heuristically, observe that the fractional Laplacian has Fourier symbol $ -|u|^{2\nu}$. The Caputo derivative of equation \eqref{funzione caratteristica diffusione anomala} is
$$ \frac{\partial ^\nu}{\partial t ^\nu}  \varphi_\nu^1 (u) = -\frac{1}{2^\nu} |u|^{2\nu}  \varphi_\nu^1 (u)  $$
and Fourier inversion gives equation \eqref{fractional diffusion equation}.
We will make this rigorous in the next section (see Remark \ref{fradens}).
\end{os}
In the following proposition we shall see that the Mittag-Leffler anomalous-diffusion process is equal in distribution to a randomly scaled Brownian motion.  We emphasize once more that the anomalous diffusion under consideration is not defined as a time-changed Brownian motion. Indeed, it is presented as a specific case of non-Markovian continuous-time process and Definition \ref{Definition: anomalous diffusion} is internally consistent. Nevertheless, the following result is useful as it renders the process analytically tractable. 

\begin{prop} \label{Convergence model 1: ugualianza fdd del mb lampertizzato}
    Let $B = \{B_t,\ t\geq0\}$ be a standard Brownian motion in $\mathbb{R}^d$, $d\geq 1$, with covariance matrix given by Equation \eqref{Anomalous diff 1: forma del j blocco nella definizione}.
    Then $ \{W_t\} \overset{fdd}{=} \{B_{\mathcal{L}t}\}$, where $\mathcal{L}$ is a Lamperti random variable independent of $B$.
\end{prop}
\begin{proof}
    Let us consider a sequence of times $0 \leq t_1 <\ldots < t_n, \ n\ge1$. Let $\mathcal{L}$ follow a Lamperti distribution of parameter $\nu\in(0,1]$. Consider the vector $\widetilde{\Gamma} := (B_{\Lam t_1},\ldots, B_{\Lam t_n})$, on some probability space, with random var-cov matrix given by $\Lam Q$. Then, using $\E$ for the expectation on this probability space, one can write
    \begin{align*}
        \E \lq e^{i \langle u, \widetilde{\Gamma}\rangle} \rq 
        &= \E\lq e^{-\frac{1}{2} \langle u, \Lam Q u \rangle} \rq \\
        &= \E\lq e^{-\frac{1}{2}\langle u, Q u \rangle \Lam} \rq = \mathcal{M}_{\nu}\l - \l \frac{1}{2}\langle u, Q u \rangle \r^{\nu} \r \qquad \forall u\in \R^{nd},
    \end{align*}
    where the matrix $Q\in \R^{nd\times nd}$ is a block diagonal matrix,  and the $j$th block has the form as in Equation \eqref{Anomalous diff 1: forma del j blocco nella definizione}. The thesis follows.
\end{proof}

As a consequence, if $Q$ is positive definite, then the process $\Gamma$ as in Definition \ref{Definition: anomalous diffusion} has density given by
\begin{align*}
    f_{\Gamma}(x) &= \int_0^{\infty} \frac{1}{\l 2\pi\, l^{nd} \det{Q}\r^{nd/2}} e^{-\frac{1}{2l}\langle x,Q^{-1}x\rangle} \mathpzc{l}(dl) .
\end{align*}

In the following theorem we consider analogous assumptions to those considered in Proposition \ref{limite diffusivo classico}, which deals with the Markovian case. Specifically, they involve letting the speed $c$ and the parameter $\lambda$ tend to infinity in such a way that the ratio $c^2/\lambda$ remains finite.
However, as noticed in Section \ref{Exchengeable Fractional PP}, the interpretation of  $\lambda$ differs in the present framework: indeed, for $\nu \in (0,1)$, the parameter $\lambda$ characterizes the intensity of the counting process as described by Equation \eqref{comportamento locale}, whereas for $\nu = 1$, corresponding to the Markovian case, $\lambda$ also represents  the expected number of direction changes within a unit time interval.

\begin{te} \label{Diffusion: convergence theorem convergenza diffusione anomala}
Let $W$ be a Mittag-Leffler anomalous-diffusion process as in Definition \ref{Definition: anomalous diffusion} and let $X_{(1)}^{\nu}$ be the random flight defined in Section \ref{Subsection: Random flight Model 1}. Then, under the scaling limit
\begin{align} \label{ppp}
    c\to \infty \qquad \lambda \to \infty \qquad \frac{c^2}{\lambda} = 1 
\end{align}
we have that
\begin{align*}
     \lg X_{(1)}^{\nu}(t)\rg \overset {fdd}{\longrightarrow}\{W_t\} .
\end{align*}
\end{te}

\begin{proof}
Without loss of generality, we shall put $c^2 = \lambda $ and consider the one variable limit as $c\to\infty$. 
The proof is based on the representation in Theorem \ref{Model 1: teorema su Duhamel}. First, we observe that 
\begin{align*}
    \mathds{P}^{(x,v)}\l X_{(1)}^{\nu}(t_1) \in A_1,\ldots, X_{(1)}^{\nu}(t_n) \in A_n \r = \int_{0}^{\infty}\mathds{P}^{(x,v)} \l X^{(c,l\lambda)}_{t_1}\in A_1,\ldots, X^{(c,l\lambda)}_{t_n}\in A_n \r \mathpzc{l}(dl)
\end{align*}
where $A_i$ are continuity sets, $i=1,\ldots,n$, $n\geq1$. We now apply the diffusive limit \eqref{ppp} to both members of the above equation. By the dominated convergence theorem, 
\begin{align*}
   & \lim_{c\to\infty} \mathds{P}^{(x,v)}\l X_{(1)}^{\nu}(t_1) \in A_1,\ldots, X_{(1)}^{\nu}(t_n) \in A_n \r \notag \\ &= \lim_{c\to\infty} \int_{0}^{\infty}\mathds{P}^{(x,v)} \l X^{(c,l\lambda)}_{t_1}\in A_1,\ldots, X^{(c,l\lambda)}_{t_n}\in A_n \r \mathpzc{l}(dl) \\
    &= \int_{0}^{\infty} \lq \lim_{c\to\infty} \mathds{P}^{(x,v)} \l X^{(c,l\lambda)}_{t_1}\in A_1,\ldots, X^{(c,l\lambda)}_{t_n}\in A_n \r \rq \mathpzc{l}(dl) \\
    &= \int_{0}^{\infty} \P^x\l B_{\frac{t_1}{l}} \in A_1,\ldots, B_{\frac{t_n}{l}}  \in A_n \r \mathpzc{l}(dl) \\
    &=  \P^x\l B_{\frac{t_1}{\Lam}} \in A_1,\ldots, B_{\frac{t_n}{\Lam}}  \in A_n \r
\end{align*}
where we used the Proposition \ref{limite diffusivo classico} with $D = 1$.
Using the property of the Lamperti distribution in Equation \eqref{Lamperti: L = 1/L} we have 
\begin{align*}
    \P^x\l B_{\frac{t_1}{\Lam}} \in A_1,\ldots, B_{\frac{t_n}{\Lam}}  \in A_n \r \overset{}{=} \P^x\l B_{t_1\Lam} \in A_1,\ldots, B_{t_n\Lam}  \in A_n \r.
\end{align*}
Finally, by Proposition \ref{Convergence model 1: ugualianza fdd del mb lampertizzato} we get 
\begin{align*}
    \lim_{c\to\infty} \mathds{P}^{(x,v)}\l X_{(1)}^{\nu}(t_1) \in A_1,\ldots, X_{(1)}^{\nu}(t_n) \in A_n \r = \P^x\l W_{t_1} \in A_1,\ldots, W_{t_n} \in A_n \r.
\end{align*}
\end{proof}

\section{Averaging Feller Semigroups}
\label{Section: averaging feller semigroups}

Both  the anomalous diffusion defined in Section \ref{Diffusion model 1} and the para-Markov chains recalled in Section \ref{Basic notions and preliminary results} are governed by equations of the form   $ \partial _t^\nu q= -(-G)^\nu q$.
 One can conjecture that this theory can be extended to more general cases. In particular, we might think that, whenever a non-Markovian process $X = \{X_t,\ t\geq0\}$ is equal in distribution to $ \lg M_{\Lam t},\ t\geq0\rg$, where $M = \{M_t,\ t\geq0\}$ is a Markovian process  and $\Lam$ is a Lamperti random variable, then its governing equation is of the type $ \partial _t^\nu q= -(-G)^\nu q$, where $G$ is the infinitesimal generator of $M$.
In this section we make this idea rigorous and we present the connection between random scaling of Markov processes and non-local operators. 
This connection, which is interesting in itself, will be used later in the paper.

Let us consider a Polish space $(E,\mathcal{E})$ endowed by its Borel $\sigma$-algebra.
Let $\l \Omega, \mathcal{F}, \mathbb{P}^x \r$, $x\in E$, be a family of probability spaces equipped with a filtration $\lg \mathcal{F}_t,\ t\geq0 \rg$. Let $M = \lg M_t,\ t\geq0\rg$ be an adapted Feller process, which takes value on $(E,\mathcal{E})$, associated with the Feller semigroup of operators
\begin{align}
\sgM_t h(x)= \mathbb{E}^x h(M_t) \notag 
\end{align}
on the Banach space $\l C_0(E), \left\| \cdot \right\| \r$,
endowed with the sup-norm. Feller semigroups are strongly continuous, i.e. $\left\| \sgM_t h -h \right\| \to 0$ as $t \to 0$ (for more on Feller semigroups see \cite{schillinglevy}). It is known that the mapping $t \mapsto \sgM_t  h$ is the unique solution to the abstract Cauchy problem
$$ \frac{\partial }{\partial  t} g(t) = Gg(t) \qquad g(0)=h \qquad h\in \Dom (G)$$
where $G$ is the infinitesimal generator of the semigroup.

We here consider the operators $\{ \sgL_t\}_{t \geq 0}$ defined through the Bochner integrals
 \begin{align}
     \sgL_t h \, = \, \int_0^{\infty} \sgM_{ty} h \, \mathpzc{l}(y)dy
     \label{bb}
 \end{align}
 where $\mathpzc{l}(\cdot)$ denotes a density of a Lamperti distribution as in Equation \eqref{Misura Lamperti distribution}.



We now show that the mapping  $t \mapsto \sgL_th$ is the unique solution of the following (non-local) abstract Cauchy problem
\begin{align}
\frac{\partial ^\nu}{\partial t^\nu} g(t) = -(-G)^\nu g(t) \qquad g(0)=h \in \Dom(G),\notag 
\end{align}
where the operators are defined as follows. The fractional derivative on the left-hand side is
\begin{align}
\frac{\partial^\nu}{\partial t^\nu} g(t) := \, \frac{1}{\Gamma (1-\nu)}\frac{\partial }{\partial t} \int_0^t (g(s)-g(0)) (t-s)^{-\nu} ds  
\label{deffrac}
\end{align}
for any $g$ such that the integral makes sense as a Bochner integral on $C_0 (E)$ and the differentiation is possible in the strong sense. The operator defined in Equation \eqref{deffrac} is the abstract analogue of the Caputo fractional derivative in Equation \eqref{Definizione di Derivata di Caputo}. In order to define the operator $- \l -G \r^\nu$ we use Bochner subordination technique, as follows. We define the family of operators $\lg \sub_t \rg_{t \geq 0}$ such that
\begin{align}
\sub_t h := \int_{0}^{+\infty} \sgM_sh \,  \mu_t(s)ds
\label{subsem}
\end{align}
where $\mu_t(\cdot)$ is the density of a stable subordinator $\{H(t),\ t\geq 0\}$ of stability index  $\nu \in (0,1)$, i.e., a strictly increasing L\'evy process with Laplace transform
\begin{align}
    e^{-t\lambda^\nu} \, = \, \int_0^{+\infty} e^{-\lambda s} \, \mu_t(s) ds.
    \label{laplacestable}
\end{align}
It is known that the family $\lg \sub_t \rg_{t\geq0}$ forms a strongly continuous semigroup on $C_0(E)$ \cite[Proposition 13.1]{librobern} and that it is generated by the operator $G^{(\nu)}$ such that \cite[Theorem 13.6]{librobern}
\begin{align}
    G^{(\nu)} h := \int_0^{+\infty} \l \sgM_s h - h \r \frac{\nu s^{-\nu-1}}{\Gamma (1-\nu)}ds , \qquad h\in \Dom(G),
\label{defGnu}
\end{align}
a Bochner integral on $C_0(E)$. It is further true that (see \cite{schillingpaper})
\begin{align}
  G^{(\nu)} \mid_{\text{Dom}(G)}  = -\l -G \r^\nu.
\end{align}
In the following, we shall denote  by 
$
\{L(t),\ t\geq 0\}
$
the inverse process of $H$,  i.e.  
\begin{align}
L(t) := \inf \ll s \geq 0 : H(s) > t \rr \qquad t\geq 0  \label{definizione inverso}.
\end{align}
The random variable $ L(t) $ is known to be absolutely continuous for each $t>0$ and belongs to the class of inverse subordinators, crucial processes in the context of anomalous diffusion related to fractional kinetic (see more on inverse subordinators in \cite{tlms}).

\begin{te} \label{teorema equazione}
    Let $\lg \sgM_t \rg_{t\geq0} $ be a Feller semigroup as above, generated by $\l G, \Dom(G) \r$. Consider the operators $\lg \sgL_t \rg_{t\geq0}$ defined in Equation \eqref{bb} and the abstract Cauchy problem
        \begin{align}
            \frac{\partial^{\nu}}{\partial t^\nu} g(t) = -(-G)^\nu g(t) \qquad g(0)=h,
        \label{eqgen}
    \end{align}
    where the operators appearing in Equation \eqref{eqgen} are defined as in Equations \eqref{deffrac} and \eqref{defGnu}. It is true that: 
    \begin{enumerate}
        \item $\lg \sgL_t \rg_{t\geq0}$ is a family of uniformly bounded linear operators on $C_0(E)$ such that
        \begin{align*}
            \sgL_t : C_0(E) \mapsto C_0 (E) \qquad \sgL_t : \text{Dom}(G) \mapsto \text{Dom}(G)
        \end{align*}
        for any $t \geq 0$;
        \item For any $h \in C_0(E)$ the mapping $t \mapsto \mathcal{Q}_th$ is strongly continuous, i.e., $\left\| \sgL_th - \sgL_sh \right\| \to 0$, as $t \to s$ for any $s,t \geq 0$; 
        \item For any $h \in \text{Dom}(G)$ the mapping $t \mapsto \sgL_t h$ solves Equation \eqref{eqgen}.
    \item If $\mathpzc{q}(t)$ is a bounded (strongly) continuous solution on $C_0(E)$ to Equation \eqref{eqgen} such that $t \mapsto -(-G)^\nu \mathpzc{q}(t)$ is bounded and (strongly) continuous, then $\mathpzc{q}(t) = \mathcal{Q}_th$ for any $t \geq 0$.
    \end{enumerate}
\end{te}

\begin{proof}
For the proof the following representation for $\sgL_t$:
\begin{align}
\sgL_t h \, = \, \int_0^{+\infty} \sub_s h \,  \lambda_t(s)ds
\label{repr}
\end{align}
is used, where $\lambda_t(\cdot)$ is the density of the inverse stable subordinator defined in Equation \eqref{definizione inverso}
and $\sub_s$ is defined in Equation \eqref{subsem}. In order to derive representation \eqref{repr} we use the following arguments. Consider $H_i := \lg H_i(t),\ t\geq 0 \rg$, $i=1,2$, two i.i.d. stable subordinators, having Laplace transform as in Equation \eqref{laplacestable}   and let us use the notation $\lg L_2(t),\ t\geq 0 \rg$ for the inverse process of $H_2$. 
Note that 
\begin{align*}
\mathbb{E}e^{-\eta H_1(L_2(t))}& = \mathbb{E} e^{-\eta ^\nu L_2(t)}= \mathcal{M}_{\nu}(-t^\nu \eta ^\nu)  \qquad \eta \geq 0,
\end{align*}
where we used the Laplace transform in Equation \eqref{laplacestable} and the fact that (see Remark 3.1 in \cite{MeerJap})

\begin{align}
    \int_0^{+\infty} e^{- \gamma s} \lambda_t(ds) \, = \, \mathcal{M}_\nu \l - \gamma  t^\nu \r, \qquad \gamma \geq 0.
\end{align}
We further observe that the Definition \ref{Definition of Lamperti distribution} of Lamperti distribution  gives 
\begin{align*}
    \mathbb{E} e^{-\eta t \mathcal{L}}&= \mathcal{M}_\nu(-t^\nu \eta ^\nu), \qquad \eta \geq0,
\end{align*}
so that
\begin{align}
    t\mathcal{L}\overset{d}{=} H_1(L_2(t)).
\end{align}
We now use the notation $\mathpzc{g}_t(\cdot)$ for the density of $H_1(L_2(t))$, i.e.
\begin{align}
    \mathpzc{g}_t(z) \, := \, \int_0^{+\infty} \mu_s(z) \lambda_t(s)ds \qquad z>0. \label{densità del subordinatore subordinato}
\end{align}

It follows that
\begin{align}
    \sgL_t h&= \int _0^{\infty} \sgM_{ty}h \, \mathpzc{l}(y)dy     \nonumber \\
    &=   \int _0^{\infty} \sgM_{z}h \, \mathpzc{g}_t(z)dz\notag \\
    &= \int _0^{\infty} \sub_{s}h \, \lambda_t(s)ds
\end{align}
where in the last step we used Equation \eqref{densità del subordinatore subordinato}, the definition of $\sub_t$ in Equation \eqref{subsem} and the Fubini theorem for Bochner integrals (see \cite[Theorem 1.1.9]{abhn}).

We are now ready to prove the Theorem.
\begin{enumerate}
    \item 
        The linearity simply follows from the definition \eqref{bb}. Similarly, the fact that $\sgL_t h\in C_0(E)$ for each $h\in C_0(E)$ and $t\geq0$ comes from the definition of Bochner integral provided that the integral exists. This can be  seen by using \cite[Theorem 1.1.4]{abhn} that applies here since $\left\| \mathcal{T}_{ty}h  \right\| \leq \left\| h \right\|$. The same argument leads to contractivity, indeed,
       \begin{align}
        \left\|  \sgL_t h \right\| \, =\, & \left\| \int_0^{+\infty} \sgM_{ty} h \mathpzc{l}(dy)  \right\| \notag \\
        \leq  \, & \int_0^{+\infty} \left\| \sgM_{ty}h  \right\| \, \mathpzc{l}(dy) \notag \\
        \leq \, & \left\| h \right\|, 
    \end{align}
   which implies that $\lg \sgL_t \rg_{t\geq0}$ is uniformly bounded. 
   So, the operators $\lg \sgL_t \rg$ are also well-defined and contractive.

    We have that $\lg\sgM_t \rg_{t \geq 0}$ is a Feller semigroup, therefore we can apply \cite[Lemma 4.5]{schillinglevy} to conclude that the subordinate semigroup $\lg \sub_t \rg_{t \geq 0}$ is also a Feller semigroup. It follows that $\sub_t$ maps $C_0(E)$ and $\text{Dom}(G)$ into themselves and it is strongly continuous. 
    
    We now prove that $\sgL_t h \in \Dom (G)$ if $h\in \Dom (G)$.  We shall use that $\text{Dom}(G)\subseteq \text{Dom}\l G^{(\nu)} \r $ and that $G^{(\nu)} \sub_y h = \sub_y G^{(\nu)} h$ for $h \in \text{Dom}(G)$. To get this result, it is sufficient to check the assumptions of \cite[Proposition 1.1.7]{abhn}. Specifically, one has that $\sgL_t h= \int _0^\infty \sub_sh \lambda _t(s) ds$ is well defined, i.e. 
\begin{align*}
    s \mapsto \sub_sh \lambda _t(s)
\end{align*}
is Bochner integrable for any $t>0$. Moreover, $\sub_sh \lambda _t(s)$    falls in the domain of $G^{(\nu)}$ because $G^{(\nu)}$ is the generator of $\sub_s$. 
Furthermore, we know that $G^{(\nu)}\sub_sh \lambda _t(s)= \sub_s (G^{(\nu)}h) \lambda _t(s) $ and $\sub_s (G^{(\nu)}h) \lambda _t(s)$ is Bochner integrable so that 
\begin{align*}
    s\mapsto G^{(\nu)}\sub_sh \lambda _t(s)
\end{align*}
is Bochner integrable as well, and the first item of the thesis follows. 

    \item Now we prove strong continuity. Again, the subordinate semigroup $\lg \sub_t \rg_{t \geq 0}$ is a Feller semigroup and it is strongly continuous. Furthermore, since $L_2(t)$ is the inverse of a strictly increasing subordinator, then it is path-continuous and then $\lg \sub_{L_2(t)} \rg_{t\geq0}$ is strongly continuous.
    
    By the representation \eqref{repr} we have that
    \begin{align}
        \sgL_t h \, = \, \mathbb{E} \sub_{L_2(t)}h,
        \notag 
    \end{align}
    which gives
    \begin{align*}
        || \sgL_{t_2} h - \sgL_{t_1} h || \leq \E || \sub_{L_2(t_2)}h  - \sub_{L_2(t_1)}h|| \qquad t_1,t_2 \geq0.
    \end{align*}
    Consider  the limit $t_1 \to t_2$. Using $\left\| \sub_{L_2(t) }h \right\| \leq \left\| h \right\|$ to apply the dominated convergence for Bochner integrals (\cite[Theorem 1.1.8]{abhn}), we get strong continuity.

    \item 
        As a consequence of \cite[Theorem 3.1]{fracCauchy} a Feller semigroup $\lg \sub_{t} \rg_{t\geq0}$ generated by $G^{(\nu)}\mid_{\text{Dom}(G)} = - \l -G \r^\nu$, integrated against the distribution of an inverse stable subordinator defines a mapping that solves Equation \eqref{eqgen}.
    Consistently, from Items (1) and (2) above, the functions 
    \begin{align*}
        t\mapsto \frac{\partial ^{\nu}}{\partial t^{\nu}} \sgL_t h\ \qquad \text{and} \qquad t\mapsto -(-G)^{\nu}\sgL_t h
    \end{align*}
    are well defined in the strong sense. 
    \item First, we prove that $t\longmapsto G^{(\nu)}\mathcal{Q}_th$ is bounded and (strongly) continuous for each $h\in \Dom(G)$. Indeed,
    \begin{align}
       G^{(\nu)} \mathcal{Q}_t h \, = \, \int_0^{+\infty} \mathcal{U}_s G^{(\nu)} h \lambda_t(s) ds \, = \, \mathcal{Q}_t G^{(\nu)} h
    \end{align}
    by the same arguments used in the proof of Item (1) and the strong continuity follows from the strong continuity and boundedness of $t \mapsto \mathcal{Q}_t G^{(\nu)} h$, since $G^{(\nu)} h \in C_0(E)$.
    
    Uniqueness follows from a Laplace transform argument. Note that $v(t) = \mathpzc{q}(t) - \mathcal{Q}_th$ solves the problem \eqref{eqgen} with $v(0)=0$. We can take the Laplace transform, $t \mapsto s$, of Equation \eqref{eqgen}, which is well defined by the boundedness of both sides. 
    
    Denoting by $\widetilde{v}$ the Laplace transform of $v$, one gets
    \begin{align}
        s^{\alpha} \widetilde{v}(s) = G^{(\nu)} \widetilde{v}(s) \iff \l s^\alpha -G^{(\nu)} \r \widetilde{v}(s) = 0.
    \end{align}
    Since the resolvent $\l s^\alpha-G^{(\nu)} \r^{-1}$ exists and is bounded, we have $\widetilde{v}(s)=0$ for any $s \geq 0$ being $\widetilde{v}(s)$ continuous.
\end{enumerate}

\end{proof}

In the next section we shall apply the theory of averaging Feller semigroups to transport processes. Before doing that, in the following remark we clarify the statement of Remark \ref{Anomalous diffusion: remark equazione governante}.
\begin{os}
    \label{fradens}
    From Proposition \ref{Convergence model 1: ugualianza fdd del mb lampertizzato} and the previous Theorem, the process $W_t$ is such that the function $q(x,t) = \mathds{E}^xu\l W_t\r$ solves a fractional kinetic equation of the form of Equation \eqref{eqgen}. In this case, the corresponding Feller process is a Brownian motion whose generator is such that $G|_{C_0^2 (\mathbb{R}^d)} = \Delta$, i.e., the Laplace operator. Therefore \eqref{eqgen} reduces to a fractional diffusion equation in both time and space. It follows that
    \begin{align}
        q(x,t) = \int_{\mathbb{R}^d} u(y) p_t(y-x)dy,
    \end{align}
    where $p_t(y-x)$ is a density for the random variable $W_t$, satisfies Equation \eqref{eqgen} with $G^\nu|_{C_0^2 (\mathbb{R}^d)} = -(-\Delta)^\nu$. Furthermore, $p_t(y-x) dy \to \delta_x(dy)$ weakly as $t \to 0$ and thus we can conclude that $p_t(y-x)$ satisfies the equation in the sense of fundamental solutions.
\end{os}

\section{Motion among random obstacles (Model 2)} \label{section model 2}

\subsection{Transport  with infinite mean flight times and random speed} \label{Subsection: Random flight Model 2}
We here consider another non-Markovian,  isotropic transport process, which differs from Model 1 because the speed of the particle is not fixed anymore. 
Specifically, we here assume that the speed is a random variable which is stochastically dependent on the random flight times. Again, the random number of direction changes up to time $t$ is $\ef_t$, i.e. the exchangeable fractional Poisson process defined in Section \ref{Exchengeable Fractional PP}. 
In more detail, the model is defined by the following assumptions.

\begin{assumption}{A2}{}\label{Isotropic: assumptions Model 2}
Suppose that in Definition \ref{Definizione isotropic transport process generica} the vectors $\l \mathcal{C}, J_1^{(\nu)},\ldots, J_n^{(\nu)}\r$, $n \geq1$, under all the probability measures $\mathds{P}^{(x,v)}$, have density given by
    \begin{align}
        \mathds{P}^{(x,v)}\lq \mathcal{C}\in d\tau, J_1^{(\nu)}\in dx_1,\ldots, J_n^{(\nu)} \in dx_n \rq = \l \lambda \frac{\tau}{c} \r^n e^{-\lambda \frac{\tau}{c} \sum_{i = 1}^nx_i}\ \mathpzc{l}\l \frac{\tau}{c}\r \frac{1}{c}\ d\tau dx_1\ldots dx_n \label{Random flight 2: densità congiunta C e Js}
    \end{align}
    where $\mathpzc{l}(\cdot)$ denotes the Lamperti distribution as in Equation \eqref{Misura Lamperti distribution}, which is equivalent to say that
    \begin{align}
        \l \mathcal{C}, J_1^{(\nu)}, \dots, J_n^{(\nu)} \r \overset{d}{=} \l  c \Lam , \frac{1}{\mathcal{L}}J_1, \dots, \frac{1}{\mathcal{L}}J_n \r \label{Random flight 2: uguaglianza in distribuzione C e Js}
    \end{align}
    where $\{J_n,\ n\geq 1\}$ are i.i.d. exponential variables of mean $1/ \lambda$.
    \end{assumption}
This assumption has the following implications: 
    \begin{enumerate}
    \item The joint distribution of the flight times $\lg J_n^{(\nu)},\ n\geq1 \rg$ is given by Equation \eqref{exchangeable fractional Poisson: joint survival function for (J_1, ..., J_n)} for some $\nu\in(0,1]$.
    \item The speed is $\mathcal{C} =  c \Lam $ where $c$ is a positive constant and $\mathcal{L}$ follows a Lamperti distribution of parameter $\nu$, as in Definition \ref{Definition of Lamperti distribution}.
    \item  The speed increases with $\Lam$ but at the same time each flight time decreases, such that the free path remains unchanged.
\end{enumerate}
We shall denote the random flight process of Assumptions \ref{Isotropic: assumptions Model 2} by
\begin{align*}
        \l \Xcl,  \Vcl\r,  \qquad \nu \in (0,1].
\end{align*}
In this setting, the flight times $\lg J_n^{(\nu)},\ n\geq1 \rg$ are dependent and they all have the same Mittag-Leffler distribution, so that $\mathbb{E}J_n^{(\nu)} =\infty$ for each $n\geq1$. Note that for $\nu = 1$ one re-obtains the isotropic Markovian random flight of Section \ref{Markovian random flight}.

The process is non-Markovian but, as shown in the following theorem, it is equal in distribution to a mixture of Markovian random flights of the type defined in Section \ref{Markovian random flight}.
In the same theorem we also determine the governing equation of $\l \Xcl, \Vcl \r$ which involves fractional power of the operators appearing in the linear Boltzmann equation \eqref{Boltzmann equation}.

We will use the family of operators $\lg\opmoddue_t\rg_{t\geq0}$ defined by
\begin{align}
    \opmoddue_t h (x,v) &:= \mathds{E}^{(x,v)} \lq h \l \Xcl,  \Vcl\r  \rq \qquad h \in  C_0\l \mathbb{R}^d \times S^{d-1} \r. \label{ooo} 
\end{align}
Specifically, we will consider the Boltzmann semigroup $\lg \sgM_t^{(c,\lambda)} \rg_{t\geq0}$ as defined in Equation \eqref{semigruppo Boltzmann} with representation as in Theorem \ref{Isotropic: teorema rappresentazione semigrouppo Boltzmann}. Then, by applying the Phillips formula \eqref{defGnu} to it, we are able to properly define the fractional power of the Boltzmann generator \eqref{Boltzmann generator} as

\begin{align} \label{Isotrpoic: fractional Boltzmann operator modello 2}
    -\bigg ( -c\, v\cdot \nabla -\lambda (L-\mathcal{I}) \bigg )^\nu h := \int_0^{\infty} \l \sgM_s^{(c,\lambda)} h - h \r \frac{\nu s^{-\nu - 1}}{\Gamma(1 - \nu)}ds, \qquad h\in D,
\end{align}
where $D$ is the domain of the Boltzmann generator as in Equation \eqref{Isotropic: dominio del generatore del semigruppo markoviano}.

\begin{te} \label{Model 2: teorema equazione Boltzmann frazionaria e uguaglianza in distribuzione}
    Let  $\l X^{(c,\lambda)}_t, V^{(c ,\lambda)}_t \r$ be the Markovian random flight defined in Section  \ref{Markovian random flight} and let $\mathcal{L}$ be a Lamperti random variable. Then
    \begin{align} \label{mixture Model 2}
        \l \Xcl,  \Vcl \r \overset{\text{d}}{=} \l X^{(c ,\lambda)}_{\mathcal{L}t}, V^{(c ,\lambda ) }_ { \mathcal{L}  t} \r.
    \end{align}
   Moreover, for $h \in  \text{D}$,   the function $ t \mapsto \opmoddue_t h$ defined in Equation \eqref{ooo} is the unique solution of the (abstract) fractional Boltzmann equation
    \begin{align}                           \frac{\partial ^\nu}{\partial   t^\nu} g(t) = -\bigg (-c\, v\cdot \nabla  -\lambda (L-\mathcal{I}) \bigg )^\nu g(t)   \notag 
    \end{align}
    subject to $g(0)= h$, in the sense of Theorem \ref{teorema equazione}, Item (4).
\end{te}
\begin{proof}
Under the Assumptions \ref{Isotropic: assumptions Model 2}, the isotropic transport process in Definition \ref{Definizione isotropic transport process generica} reads

\begin{align*}
    \l \begin{matrix}
        \Vcl \\  \Xcl
    \end{matrix} \r &= 
    \l \begin{matrix}
        v_m \\  
        x + \sum_{k = 1}^{m} \mathcal{C}v_{k-1}J_k^{(\nu)} + \mathcal{C} v_m\l t - \sum_{k = 1}^{m} J_k^{(\nu)} \r
    \end{matrix} \r
    \qquad
    \sum_{k = 1}^{m} J_k^{(\nu)} \leq t < \sum_{k = 1}^{m+1} J_k^{(\nu)}
\end{align*}
for $m \geq 0$,
which can compactly be re-written as
\begin{align*}
\l \begin{matrix}
        \Vcl \\  \Xcl
    \end{matrix} \r  = 
    \sum_{m = 0}^{+\infty} \l \begin{matrix}
        v_m \\  
        x + \sum_{k = 1}^{m} \mathcal{C}v_{k-1}J_k^{(\nu)} + \mathcal{C} v_m\l t - \sum_{k = 1}^{m} J_k^{(\nu)} \r
    \end{matrix} \r 
    \mathds{1}_{\lg \sum_{k = 1}^{m} J_k^{(\nu)} \leq t < \sum_{k = 1}^{m+1} J_k^{(\nu)} \rg}.
\end{align*}

From Equation \eqref{Random flight 2: uguaglianza in distribuzione C e Js} we have, for each $t\geq0$,
\begin{align*}
    \l \begin{matrix}
        \Vcl \\  \Xcl
    \end{matrix} \r \overset{d}{=} 
    \l \begin{matrix}
        v_m \\  
        x + \sum_{k = 1}^{m} cv_{k-1}J_k + c v_m\l \mathcal{L}t - \sum_{k = 1}^{m} J_k \r
    \end{matrix} \r
    \qquad
    \sum_{k = 1}^{m} J_k \leq \mathcal{L}t < \sum_{k = 1}^{m+1} J_k
\end{align*} 
and the equality in distribution of Equation \eqref{mixture Model 2} follows. 

Hence we can re-write \eqref{ooo} as 
\begin{align*}
    \opmoddue_t h (x,v)&= \int _0^{\infty} \mathds{E}^{(x,v)} \lq h  \l X^{(c,\lambda )}_{lt}, V^{(c,\lambda)}_{lt} \r \rq \, \mathpzc{l}(dl)    = \int _0^{\infty} \sgM^{(c, \lambda )}_{lt} h (x,v) \,  \mathpzc{l}(dl)
\end{align*}
where $\sgM_t ^{(c,\lambda)}$ is the Markovian semigroup defined in Equation \eqref{semigruppo Boltzmann}. We recall that the family $\lg \sgM_t^{(c, \lambda)} \rg _{t\geq0}$ is a Feller semigroup, whose generator  \eqref{Boltzmann generator} has domain $D$, hence $\opmoddue_t$ is an averaged Feller semigroup of the form of Equation \eqref{bb}.
We fall under the assumptions of Theorem \ref{teorema equazione} from which the result immediately follows.
\end{proof}

\begin{os} \label{Model 2: Remark sull'uguaglianza in distribuzione del ranom flight}
    By using analogous arguments as in the above proof, recalling Equation \eqref{Random flight 2: uguaglianza in distribuzione C e Js} and observing that each $ \frac{1}{\Lam} J_i$  follows an exponential distribution of random mean $1 / (\lambda \Lam)$,  we also get
 \begin{align} 
        \l \Xcl,  \Vcl \r \overset{\text{d}}{=} \l X^{(c \mathcal{L},\lambda \mathcal{L})}_t, V^{(c \mathcal{L},\lambda \mathcal{L})}_t \r. \notag
    \end{align}
    The dependence of Equation \eqref{Random flight 2: densità congiunta C e Js} guarantees that each flight time is inversely proportional to the speed; therefore, the length of each free path remains constant under the scaling with the Lamperti variable.
\end{os}

\begin{os}
    Under Assumptions \ref{Isotropic: assumptions Model 1}, the random variable $\Xl$ is supported on a ball centered at the starting position $x$ of radius $ct$, and thus the mean square displacement, computed in  Proposition \ref{Model 1: proposition on mean squared  displacement}, is finite. 
        Instead, in the case of  Assumptions \ref{Isotropic: assumptions Model 2} the speed is random with a distribution supported on $(0 +\infty)$, and thus the radius of the ball can assume any positive value. It follows that the random position $\Xcl$ is not supported anymore on a set with finite measure. Furthermore, from the  mean squared  displacement of the Markovian random flight in Equation \eqref{Markovian: mean squared  displacement random flight} and using the representation of the process in Theorem \ref{Model 2: teorema equazione Boltzmann frazionaria e uguaglianza in distribuzione}, it is not hard to see that the random flight is superdiffusive with infinite  mean square displacement.
\end{os}

\subsubsection{Other types of fractional Boltzmann equations}
In some recent papers, other fractional Boltzmann equations have been studied. 
In \cite{RicciutiToaldo}, the authors considered  a semi-Markov transport process. Here, the flight times are i.i.d. Mittag-Leffler random variables, i.e. the counting process is a renewal process, also known as fractional Poisson process, see \cite{Beghin, Dicrescenzo, MainardiGorenfloScalas}.
Such a transport process is governed by the fractional Boltzmann equation
$$ \l \frac{\partial }{\partial t}-cv\cdot \nabla _x \r^\nu g(x,v,t) -\frac{t^{-\nu}}{\Gamma (1-\nu)}h(x+cvt,v) =  \lambda (L-\mathcal{I})  g(x,v,t)$$
under the condition $g(x,v,0)=h(x,v)$. The authors also proved that a suitable scaling limit leads to an anomalous diffusive process with continuous trajectories and finite velocity, say $\lg \widetilde{W}_t,\ t\geq 0\rg$, having a super-diffusive behavior, with the mean-squared  displacement growing as $\mathbb{E}\left| \widetilde{W}_t\right|_e^2\sim Ct^2$, $C>0$.

Another fractional Boltzmann equation with the form
\begin{align} 
\frac{\partial ^\nu}{\partial t ^\nu} g(x,v,t)=  c\, v\cdot \nabla _x  g(x,v,t)+\lambda (L-\mathcal{I}) g(x,v,t)
\end{align}
has been considered in \cite{Garra}.
In this case, the transport process is constructed as the Markovian random flight, time-changed with the inverse of a $\nu$-stable subordinator.

\subsection{Boltzmann-Grad approximation and anomalous diffusion} \label{Boltzmann-Grad approximation and anomalous diffusion model 2}

We here show that the same anomalous diffusion as in Section \ref{Diffusion model 1}, i.e., the Mittag-Leffler anomalous-diffusion process, can approximate a
different Lorentz model. We indeed  consider a Lorentz process with a Poissonian distribution of obstacle centres, defined as in Gallavotti \cite{Gallavotti}, except for the fact that the speed of the particle is here assumed to be random. Under a Boltzmann-Grad procedure, this process is proved to converge to the non-Markovian random flight defined as in Model 2. A further scaling limit leads to the Mittag-Leffler anomalous-diffusion process.

The speed of such Lorentz process is assumed to follow a Lamperti distribution as in Definition \ref{Definition of Lamperti distribution}, which is characterized by a power law decaying density. From a physical point of view this means that the gas particles are not in thermal equilibrium, in which case the speed has an exponentially decaying distribution. Indeed, it is known that a perfect gas in thermal equilibrium has a Maxwell-Boltzmann distribution for the particle velocities.

To fix the notation, consider a Poisson point process with intensity $\rho \in (0, \infty)$ for the obstacle centres. All the obstacles are spheres of radius $R$. We shall indicate with $\l \mathcal{X}^c_t,\mathcal{V}^c_t \r$ the Lorentz process as in Gallavotti \cite{Gallavotti}. The particle speed is given by the constant $c>0$, as indicated by the superscript. It is known that under the Boltzmann Grad limit
\begin{align} \label{Model 2: BG limit for Poisson}
    \rho \to \infty \qquad R \to 0 \qquad  \rho c \pi R^2 \to \lambda \in (0,\infty),
\end{align}
we have that
\begin{align}
    \l \mathcal{X}^c_t, \mathcal{V}^c_t \r   \overset{d}{\to} \l X_t^{(c, \lambda)}, V^{(c, \lambda)}_t \r \label{Model 2: Convergenza al random flight markovian Gallavotti}
\end{align}
 where $\l X^{(c, \lambda)}, V^{(c, \lambda)} \r$ is the Markovian random flight recalled in Section \ref{Markovian random flight}.

Now, consider another Lorentz process defined analogously as by Gallavotti \cite{Gallavotti}, where the particle speed is given by $\mathcal{C}= c \Lam $, being $\mathcal{L}$ a Lamperti random variable. It can be  proved that the free-flight time distribution has infinite mean analogously to what happens in Lemma \ref{Distribuzione free flight time o tempo di volo}.
For this new process, let us denote the position of the particle at time $t$ by $ \mathcal{X}_t^{(2)}$  and the unit velocity vector at time $t$, i.e. the direction of the particle at time $t$, by $ \mathcal{V}_t^{(2)}$.   We call $\l \X^{(2)}, \V^{(2)} \r := \lg \l \X^{(2)}_t, \V^{(2)}_t \r,\ t\geq0 \rg$ this Lorentz process in $\mathbb{R}^3 \times S^2$, with initial data $\l \X_0^{(2)}, \V_0^{(2)} \r$. The superscript indicates that this is the second Lorentz model considered in this paper. As in Section \ref{Sec: The Boltzmann-Grad limit modello 1} we denote 
 $P^{(x,v)} \l \cdot \r := P \l \cdot | \ \mathcal{X}_0^{(2)} = x, \mathcal{V}_0^{(2)} = v \r$. Under any $P^{(x,v)}$, we can write
\begin{align}
    \l \X^{(2)}_t, \V^{(2)}_t \r \overset{d}{=} \l \mathcal{X}^{ c \Lam }_t, \mathcal{V}^{ c \Lam }_t \r , \qquad \forall t\geq0 .\label{zzz}
\end{align}
This Lorentz process also depends on the parameters $\rho$ and $\nu$, and on the size of the obstacles $R$. Note that here the dependence on $\nu$ is only due to the random speed of the particle, being the random field Poissonian. As in the previous section, to simplify the notation, we will not make this dependence explicit.  Once again, for $\nu=1$, the speed is constant since $\mathcal{L}=1$ almost surely, and such Lorentz process coincides with that of Gallavotti's model. 
\begin{te}
Consider the Lorentz process $\l \X^{(2)}, \V^{(2)} \r$ defined above.
Under the Boltzmann-Grad limit  \ref{Model 2: BG limit for Poisson} we have \begin{align*}
    P^{(x,v)} \l \X^{(2)}_t\in dx', \V^{(2)}_t \in dv' \r \to \mathds{P}^{(x,v)} \l \Xcl  \in dx', \Vcl  \in dv' \r
\end{align*}
weakly, for all $t>0$ and for all $(x,v) \in \mathbb{R}^3 \times S^2$, where the limit process is the random flight defined in Section \ref{Subsection: Random flight Model 2}.
\end{te}
\begin{proof}
Let us write
\begin{align*}
    E^{(x,v)} h\l \X^{(2)}_t, \V^{(2)}_t \r  = \int _{S^2} \int _{B_{ct}(x)} h(x', v') P^{(x,v)} \l \X^{(2)}_t \in dx', \V^{(2)}_t \in dv' \r.
\end{align*}
where $B_{ct}(x)$ indicates the ball centered in $x$ of radius $ct$.
To get the statement, we shall prove that
\begin{align}
    \lim _{R\to 0} E^{(x,v)}h\l \X^{(2)}_t, \V^{(2)}_t \r  = E^{(x,v)} h \l X^{(c \mathcal{L},\lambda \mathcal{L})}_t, V^{(c \mathcal{L},\lambda \mathcal{L})}_t \r \qquad \forall \, \,\,h\in C_b(\mathbb{R}^3\times S^2),  \label{Model 2: Val atteso di h(pos ,vel)}
\end{align}
 where $\l X^{(c, \lambda)}, V^{(c, \lambda)} \r$ is the Markovian random flight recalled in Section \ref{Markovian random flight} and the thesis will follow from the representation of the second random flight as in Remark \ref{Model 2: Remark sull'uguaglianza in distribuzione del ranom flight}.

Specifically, we have from Equation \eqref{zzz}
\begin{align*}
    \lim _{R\to 0} E^{(x,v)}h\l \X^{(2)}_t, \V^{(2)}_t \r  &= \lim_{R\to0} E^{(x,v)}h\l \X_t^{ c \Lam }, \V_t^{ c \Lam } \r \\
    &= \lim_{R\to0} \int_0^{\infty} \mathds{E}^{(x,v)}h\l \X_t^{l c}, \V_t^{l c} \r  P\l \Lam \in dl \r.
\end{align*}
Now we observe that $h$ is bounded and so is the expected value. Using the dominated convergence theorem in the previous Equation then yields
\begin{align*}
     \int_0^{\infty} \l \lim_{R\to0} \mathds{E}^{(x,v)}h\l \X_t^{cl}, \V_t^{cl} \r \r P\l \Lam \in dl \r = \int_0^{\infty} \mathds{E}^{(x,v)}h\l X_t^{(cl, \lambda l)}, V_t^{(cl, \lambda l)} \r P\l \Lam \in dl \r
\end{align*}
where we used Equations \eqref{Model 2: BG limit for Poisson} and \eqref{Model 2: Convergenza al random flight markovian Gallavotti}, i.e. the same convergence theorem of Gallavotti's model, where the speed $c$   is replaced by $cl$ and hence $\lambda$ is replaced by $\lambda l$. The equality \eqref{Model 2: Val atteso di h(pos ,vel)} immediately follows.
\end{proof}
In the following theorem we prove that under an appropriate scaling limit the random flight of model 2 converges to the same anomalous diffusion as in Section \ref{Diffusion model 1}.

\begin{te}
Let $W$ be a Mittag-Leffler anomalous-diffusion process as in Definition \ref{Definition: anomalous diffusion}. Let $B$ be the corresponding Brownian motion, in the sense of Proposition \ref{Convergence model 1: ugualianza fdd del mb lampertizzato}. Then, under the scaling limit
\begin{align}
    c\to \infty \qquad \lambda \to \infty \qquad \frac{c^2}{\lambda} = 1 \label{Model 2: diffusive limit}
\end{align}
we have that
\begin{align*}
     \lg X_{(2)}^{\nu}(t)\rg \overset {fdd}{\longrightarrow}\{W_t\} .
\end{align*}
\end{te}

\begin{proof}
Without loss of generality, we shall put $c^2 = \lambda $ and consider the one-variable limit as $c\to\infty$.
The proof is based on the representation in Theorem \ref{Model 2: teorema equazione Boltzmann frazionaria e uguaglianza in distribuzione}. Firstly, we observe that
\begin{align*}
   \mathds{P}^{(x,v)}\l X_{(2)}^{\nu}(t_1) \in A_1,\ldots, X_{(2)}^{\nu}(t_n) \in A_n \r 
    =  \int_{0}^{\infty}\mathds{P}^{(x,v)} \l X^{(c,\lambda)}_{l t_1}\in A_1,\ldots, X^{(c,\lambda)}_{l t_n}\in A_n \r \mathpzc{l}(dl)
\end{align*}
where $A_i$ are continuity sets, $i=1,\ldots,n$, $n\geq1$. We now apply the diffusive limit \eqref{Model 2: diffusive limit} to both members of the above equation. By the dominated convergence theorem, 

\begin{align*}
  &  \lim_{c\to\infty} \mathds{P}^{(x,v)}\l X_{(2)}^{\nu}(t_1) \in A_1,\ldots, X_{(2)}^{\nu}(t_n) \in A_n \r\notag \\ = &\lim_{c\to\infty} \int_{0}^{\infty}\mathds{P}^{(x,v)} \l X^{(c,\lambda)}_{l t_1}\in A_1,\ldots, X^{(c,\lambda)}_{l t_n}\in A_n \r \mathpzc{l}(dl) \\
    =& \int_{0}^{\infty} \lq \lim_{c\to\infty} \mathds{P}^{(x,v)} \l X^{(c,\lambda)}_{l t_1}\in A_1,\ldots, X^{(c,\lambda)}_{l t_n}\in A_n \r \rq \mathpzc{l}(dl) \\
    =&  \mathbb{P}^x\l B_{t_1 \Lam} \in A_1,\ldots, B_{t_n \Lam}  \in A_n \r
\end{align*}
where we used Proposition \ref{limite diffusivo classico} with $D = 1$.
Finally, by Proposition \ref{Convergence model 1: ugualianza fdd del mb lampertizzato} we get 
\begin{align*}
    \lim_{c\to\infty} \mathds{P}^{(x,v)}\l X_{(2)}^{\nu}(t_1) \in A_1,\ldots, X_{(2)}^{\nu}(t_n) \in A_n \r = \mathbb{P}^x\l W_{t_1} \in A_1,\ldots, W_{t_1} \in A_n \r.
\end{align*}
\end{proof}


\section*{Acknowledgements}
All the authors acknowledge financial support under the
National Recovery and Resilience Plan (NRRP), Mission 4, Component 2, Investment 1.1, Call for tender No. 104 published on 2.2.2022 by the Italian Ministry
of University and Research (MUR), funded by the European Union – NextGenerationEU– Project Title “Non–Markovian Dynamics and Non-local Equations”
– 202277N5H9 - CUP: D53D23005670006 - Grant Assignment Decree No. 973
adopted on June 30, 2023, by the Italian Ministry of Ministry of University and
Research (MUR).

The author B. Toaldo would like to thank the Isaac Newton Institute for Mathematical Sciences, Cambridge, for support and hospitality during the programme Stochastic systems for anomalous diffusion, where work on this paper was undertaken. This work was supported by EPSRC grant EP/Z000580/1.

E. Scalas also acknowledges financial support under the \\ 000317\_24\_RICERCA\_UNIV\_2023\_PROG\_MEDI\_SCALAS - RICERCA ATENEO 2023 - SCALAS PROGETTI MEDI.
 The title of the project is "Approximation of stochastic processes by means of sums of random telegraph processes"

The authors are partially supported by Gruppo Nazionale per l’Analisi Matematica, la Probabilità e le loro Applicazioni (GNAMPA-INdAM).

\appendix

\section{Existence of Mittag-Leffler point process} \label{Appendix: Existence of Mittag-Leffler point process}
As a first step, we demonstrate the existence of \(\Pi\) on finite Borel subsets of \(\mathbb{R}^3\).

\begin{lem}
\label{lemmaexistfinite}
    Let $S$ be a Borel set of $\mathbb{R}^3$ with finite Lebesgue measure $|\cdot|$: $|S|<+\infty$. Then define
        \begin{align}
        \Pi_{S}  := \sum_{k=1}^{\ef} \delta_{C_k}
    \end{align}
    where $\lg C_k \rg_{k \in \lg 1, \ldots, \ef \rg}$ is a sequence of i.i.d. r.v.'s uniform on $S$ and independent of $\ef$ which, in its turn, is a r.v. with a fractional Poisson distribution
    \begin{align}
        P \l \ef = n \r \, = \, \frac{1}{n!} \l \rho |S| \r^n (-1)^n \mathcal{M}_\nu^{(n)} \l -\rho^\nu |S|^\nu \r,
    \end{align}
    where $\rho \in \mathbb{R}^+$.
 Then $\Pi_S$ is a measure as in Definition \ref{Fractional Poisson field} with $\mu(\cdot) = |\cdot \cap S|/ |S|$.
\end{lem}
\begin{proof}
   We show that $\Pi_{S}$ defined as above is a random measure that satisfies the Definition \ref{Fractional Poisson field} for $\mu = |\cdot \cap S|/|S|$. Hence, denoting by $\Pi_{S}^\prime$ a random measure as in Definition \ref{Fractional Poisson field}, we use \cite[Corollary 10.1]{kallenberg} and thus we show that
   \begin{align}
       E e^{-\Pi_{S} f} = E e^{-\Pi_{S}^\prime f} 
   \end{align}
   for any $f$ positive, continuous with compact support. We have that
   \begin{align}
       E e^{-\Pi_{S} f} \, = \, &E e^{-\sum_{k=1}^{\ef} f (C_k)} \, = \, E \l E e^{-f(C_1)} \r^{\ef} \notag \\
       = \, & \sum_m \l \frac{1}{|S|}\int_{S} e^{-f(s)} ds  \r^m P \l \ef = m \r \notag \\
       = \, & \sum_m \l \frac{1}{|S|}\int_{S} e^{-f(s)} ds  \r^m \frac{1}{m!}   \l |S|\rho \r^m (-1)^m \mathcal{M}_\nu^{(m)} \l -\l |S|\rho \r^\nu \r \notag \\
       = \, & \int_0^{+\infty} \sum_m \l \frac{1}{|S|}\int_{S} e^{-f(s)} ds  \r^m \frac{l^m(\rho|S|)^m}{m!}e^{-l\rho |S|} \mathpzc{l}(dl) \notag \\
       = \, & \int_0^{+\infty }e^{- l|S|\rho\frac{1}{|S|}\int_{S}\l 1- e^{-f(s)} \r ds } \mathpzc{l}(dl) \,
       = \,  \mathcal{M}_\nu \l -\rho^\nu \l \int_{S} \l 1-e^{-f(s)} \r ds \r^\nu \r,
   \end{align}
   where $l$ is the realisation of a Lamperti random variable $\mathcal{L}$ of parameter $\nu$.
   Now, take a simple function $f = \sum_j a_j \mathds{1}_{A_j}$. Then we have that
   \begin{align}
       E e^{-\Pi_{S}^\prime f} \, &= \, E e^{- \sum_j a_j \Pi_{S}^\prime \mathds{1}_{A_j}} \, = \, E \prod_j e^{- a_j \Pi_{S}^\prime \mathds{1}_{A_j}}  \notag \\
        &=\sum_{k_1 \cdots k_j}  \prod_j  \frac{1}{k_1! \cdots k_j!}e^{-a_j k_j } \rho^{k_j} \l |A_j \cap S|\r^{k_j} (-1)^{\sum_j k_j} \mathcal{M}_\nu^{\l \sum_j k_j \r} \l - \rho^\nu \l \sum_j \l |A_j \cap S| \r^{k_j} \r^\nu \r \notag \\
       &= \sum_{k_1 \cdots k_j} \prod_j \frac{1}{k_1! \cdots k_j!}e^{-a_j k_j} \rho^{k_j} \l |A_j \cap S| \r^{k_j} \int_0^{+\infty} e^{-\rho \sum_j \l |A_j \cap S|\r^{k_j} s } s^{\sum_jk_j} (-1)^{\sum_j k_j} \mathpzc{l}(ds) \notag \\
   & = \int_0^{+\infty} e^{-s\rho \sum_{j} |A_j \cap S| \l 1-e^{-a_j} \r} \mathpzc{l}(ds) = \mathcal{M}_\nu \l -\rho^\nu \l \sum_j |A_j \cap S| \l 1-e^{-a_j} \r \r^\nu \r \notag \\
    & = \mathcal{M}_\nu \l -\rho^\nu \l  \int_{S } \l 1-e^{-f(s)} \r ds \r^\nu \r .    \label{simplefunction}
   \end{align}
   Pick now $f$ such that $f_k \uparrow f$ for $f_k$ simple. Then, by monotone convergence we get that $\Pi_{S}^\prime f_k \uparrow \Pi_{S}^\prime f$ as well as
   \begin{align}
\int_{S} \l 1-e^{-f_k(s)} \r ds \uparrow  \int_{S} \l 1-e^{-f(s)} \r ds.
   \end{align}
Hence, by dominated convergence we obtain $ E e^{-\Pi_{S}^\prime f}= E e^{-\Pi_{S} f}$.
\end{proof}
It turns out that $\Pi$ is a Cox process in the sense of \cite[Page 180]{kallenberg}. Take a random measure on $(\mathbb{R}^3, \mathcal{B}(\mathbb{R}^3))$, say $\Xi$. A point process $\Pi$ on $\mathbb{R}^3$ is said to be a Cox process directed by $\Xi$ if it is conditionally Poisson with the intensity $\Xi$, i.e., $\mathbb{E}(\Pi \mid \Xi) = \Xi$ a.s..
\begin{lem}
\label{lemmarestriction}
Define the measure on $\mathcal{B}(\mathbb{R}^3)$ as $\Xi (\cdot)\coloneqq \mathcal{L}\rho | \cdot |$ where $\mathcal{L}$ is a Lamperti random variable and $\rho \in \mathbb{R}^+$.
A random measure on $\l \mathbb{R}^3, \mathcal{B} (\mathbb{R}^3), |\cdot| \r$ as in Definition \ref{Fractional Poisson field} on $\mathbb{R}^3$ is a Cox point process directed by $\Xi$. The restriction of $\Pi$ on a Borel set $B$ with $|B|<+\infty$ is a random measure of the same type on $S=B$ and thus it has the representation given in Lemma \ref{lemmaexistfinite}.
\end{lem}
\begin{proof}
Denote by $\mathcal{P}^\xi$ a Poisson random measure with the intensity $\xi(\cdot)$, for a measure $\xi (\cdot)$ on $\mathcal{B}(\mathbb{R}^3)$. Use \cite[Corollary 10.1]{kallenberg} to see that $\Pi$ has the same distribution as $\mathcal{P}^\Xi$ where $\Xi$ is independent of $\mathcal{P}^\xi$. Indeed note that, for any positive $f$ continuous with compact support, we have that
    \begin{align}
        E e^{-\Pi f} \, = \, E e^{-\mathcal{P}^\mathcal{L} f} \, = \, \int_0^{+\infty} E e^{-\mathcal{P}^sf} \mathpzc{l}(ds).
    \end{align}
    Observe that $\mathcal{P}^s$ is, for any $s>0$ fixed, a Poisson random measure with intensity $\rho| \cdot |$. Then
    \begin{align}
     E e^{-\Pi f} \, = \, \int_0^{\infty} e^{- l\rho \int_{\mathbb{R}^3}\l 1-e^{-f(s)} \r ds} \mathpzc{l}(d l)\, = \, \mathcal{M}_{\nu}\l -\rho^\nu \int_{\mathbb{R}^3} \l 1-e^{-f(s)} \r ds \r.
    \end{align}
    For $\Pi^\prime$ to be a measure as in Definition \ref{Fractional Poisson field} we can compute the Laplace functional for a simple function $f$ as in \eqref{simplefunction} and then the Cox property follows from monotone and dominated convergence. It follows that
    \begin{align}
        \mathbb{E} \left[ \Pi \mid \Xi \right] \, = \mathbb{E} \left[ \mathcal{P}^\Xi \mid \Xi \right]\, = \, \Xi 
    \end{align}
    where we used the independence of $\Xi$ and $\mathcal{P}^\xi$ together with the freezing Lemma (see, e.g., \cite[Lemma A.3]{schillingBM}).

    Consider now the restriction $\Pi|_B$, we have that, for any $f$ positive, continuous and compactly supported that
    \begin{align}
        E e^{-\Pi|_B f} \, = \, \mathcal{M}_\nu \l -\rho^\nu \l \int_B \l 1-e^{-f(s)} \r ds\r^\nu \r
    \end{align}
    proceeding as in \eqref{simplefunction} and using monotone and dominated convergence.
\end{proof}

\section{Details on the proof of the Boltzmann-Grad limit} \label{Appendix: appendice sulla convergenza dominata}
Following Gallavotti's intuition, the tube-like flow is the union of non-disjoint cylinders, and its volume is given by
\begin{align}
    \left| \Theta\l \underline{\widetilde{\tau}}, \underline{\widetilde{v}} \r \right| = \pi R^2ct+o(R^2). \label{Approssimazione volume del tubo di flusso}
\end{align}

We note that by Equation \eqref{BG limit}  $\rho ^n c^n \pi^n R^{2n} = \lambda^n$. Now, let us call $\Theta := \Theta\l \underline{\widetilde{\tau}}, \underline{\widetilde{v}} \r$ and

\begin{align*}
    a_n(R) :&= \int _{0< \tau _1 < \tau _2 < \dots < \tau_n < t} \int _{(S^2)^{n}}(-\lambda)^n h\l  \Xt(\underline{\widetilde{\tau}}, \underline{\widetilde{v}}),  \Vt(\underline{\widetilde{\tau}}, \underline{\widetilde{v}})\r \mathcal{M}^{(n)}_{\nu}\l- \rho^{\nu} \left| \Theta \right|^{\nu}\r \\
    &\qquad\qquad\qquad\qquad\qquad\qquad \mu(dv_1)\dots \mu(dv_n)\,  d\tau _1 \dots d\tau _n .
\end{align*}

We need it to be absolutely dominated by $b_n$ for each $n\geq 1$, the last one being independent of $R$ and summable.

Firstly, we observe that
\begin{align*}
    \left| h\l  \Xt(\underline{\widetilde{\tau}}, \underline{\widetilde{v}}),  \Vt(\underline{\widetilde{\tau}}, \underline{\widetilde{v}})\r \right| < \zeta_1 \qquad \text{for some } \zeta_1\in (0,\infty). 
\end{align*}
 
Moreover, using the notation for the derivative of the Mittag-Leffler function in Equation \eqref{Derivative of Mittag Leffler function} we get

\begin{align*}
    &\int _{0< \tau _1 < \tau _2 < \dots < \tau_n < t} \int _{(S^2)^{n}} \bigg |\mathcal{M}^{(n)}_{\nu}\l- \rho^{\nu} \left| \Theta  \right|^{\nu}\r \bigg |\, \mu(dv_1)\dots \mu(dv_n)\,  d\tau _1 \dots d\tau _n \\
    &= \int _{0< \tau _1 < \tau _2 < \dots < \tau_n < t} \int _{(S^2)^{n}} \bigg |  \l \frac{d}{d\l\rho |\Theta| \r} \r^n  \int_0^{\infty}   e^{-\rho |\Theta| l}P(\Lam \in dl) \bigg |\, \mu(dv_1)\dots \mu(dv_n)\,  d\tau _1 \dots d\tau _n.
\end{align*}

We now observe that the correction term in the approximation of the volume of the tube-like flow \eqref{Approssimazione volume del tubo di flusso} is negative. Moreover, for $R$ sufficiently small, we can write 

\begin{align*}
    \lambda t + \rho \, o \l R^2 \r \geq \frac{\lambda t}{2} \implies e^{-\l \lambda t + \rho \, o \l R^2\r  \r} \leq e^{-\frac{\lambda t}{2}}.
\end{align*}

Then, by using dominated convergence, the above equation reads
\begin{align*}
    &\int _{0< \tau _1 < \tau _2 < \dots < \tau_n < t} \int _{(S^2)^{n}} \int_0^{\infty}   l^n e^{- l \l \lambda t + \rho \, o \l R^2\r \r} P(\Lam \in dl) \mu(dv_1)\dots \mu(dv_n)\,  d\tau _1 \dots d\tau _n \\
    &\leq \int _{0< \tau _1 < \tau _2 < \dots < \tau_n < t} \int _{(S^2)^{n}}
     \int_0^{\infty}   l^n e^{-l \frac{\lambda t}{2} }  P(\Lam \in dl) \mu(dv_1)\dots \mu(dv_n)\,  d\tau _1 \dots d\tau _n \\
    &= \frac{t^n}{n!}  \int_0^{\infty}   l^n e^{-\frac{\lambda t}{2} l}  P(\Lam \in dl),
\end{align*}

where the inequality holds for sufficiently small $R$.
Finally, 

\begin{align*}
    |a_n(R)| \leq \zeta_1 \frac{1}{2^n}\int_0^{\infty} \frac{\l \lambda l t \r^n}{n!} e^{-\lambda t l} P(\Lam \in dl) =: b_n
\end{align*}
and 
\begin{align*}
    \sum_{n\geq 1}b_n = \sum_{n\geq1} \zeta_1 \frac{1}{2^n}\int_0^{\infty} \frac{\l \lambda l t \r^n}{n!} e^{-\lambda t l} P(\Lam \in dl) &\leq \zeta_1 \int_0^{\infty}   e^{-\lambda t l} \l e^{\lambda tl } - 1  \r P(\Lam \in dl)\\
    &= \zeta_1 \l 1 - \int_0^{\infty} e^{\lambda l t} P\l \Lam \in dl \r \r < \infty.
\end{align*}

\bigskip

{\bf Lorenzo Facciaroni}

Department of Statistical Sciences, Sapienza - University of Rome

Email: \texttt{lorenzo.facciaroni@uniroma1.it}

\bigskip

{\bf Costantino Ricciuti}

Department of Statistical Sciences, Sapienza - University of Rome

Email: \texttt{costantino.ricciuti@uniroma1.it} 

\bigskip

{\bf Enrico Scalas}

Department of Statistical Sciences, Sapienza - University of Rome

Email: \texttt{enrico.scalas@uniroma1.it},

\bigskip

{\bf Bruno Toaldo}

Department of Mathematics “Giuseppe Peano”, University of Turin

Email: \texttt{bruno.toaldo@unito.it}


\begin{thebibliography}{99}
\bibitem{Leonenkofield}
G. Aletti, N. Leonenko, E. Merzbach,
\textit{Fractional Poisson fields and martingales},
J. Stat. Phys. \textbf{170} (2018), 700--730.
\bibitem{Garra}
L. Angelani, A. De Gregorio, R. Garra, F. Iafrate,
\textit{Anomalous random flights and time-fractional run-and-tumble equations},
J. Stat. Phys. \textbf{191} (2024), no. 10, 129.
\bibitem{abhn}
W. Arendt, C.J.K. Batty, M. Hieber, F. Neubrander,
\textit{Vector Valued Laplace Transform and Cauchy Problem}, 2nd ed., Birkhäuser, Berlin, 2010.
\bibitem{tlms}
G. Ascione, M. Savov and B. Toaldo,
\textit{Regularity and asymptotics of densities of inverse subordinators},
Trans. Lond. Math. Soc. \textbf{11}: e70004, 2024.
\bibitem{fracCauchy}
B. Baeumer, M.M. Meerschaert,
\textit{Stochastic solutions for fractional Cauchy problems},
Fract. Calc. Appl. Anal. \textbf{4} (2001), 481--500.

\bibitem{Baleanu}
D. Baleanu, K. Diethelm, E. Scalas, J. J. Trujillo,
\textit{Fractional Calculus: Models and Numerical Methods},
Vol. 3, World Scientific, 2012.

\bibitem{balescu2005}
Balescu, R.,
\textit{Aspects of Anomalous Transport in Plasmas},
Taylor \& Francis, London, 2005.


\bibitem{bardou2001}
Bardou, F., Bouchaud, J.-P., Aspect, A., and Cohen-Tannoudji, C.,
\textit{Lévy Statistics and Laser Cooling: How Rare Events Bring Atoms to Rest},
Cambridge University Press, Cambridge, UK, 2002.

\bibitem{Barkai} E. Barkai, V. Fleurov, and J. Klafter, \textit{One-dimensional stochastic Lévy-Lorentz gas}, Physical Review E, 61 (2000), 1164.

\bibitem{barkai2000}
E.~Barkai and R.~Silbey. Distribution of variances of single molecules in a disordered lattice. 
\textit{The Journal of Physical Chemistry B}, 104(2):342--353, 2000.




\bibitem{Barlow}
R.E. Barlow, M.B. Mendel,
\textit{De Finetti-type representations for life distributions},
J. Amer. Statist. Assoc. \textbf{87} (1992), no. 420, 1116--1122.
\bibitem{Basile1}
G. Basile, A. Nota, M. Pulvirenti,
\textit{A diffusion limit for a test particle in a random distribution of scatterers},
J. Stat. Phys. \textbf{155} (2014), 1087--1111.
\bibitem{Basile2}
G. Basile, A. Nota, F. Pezzotti, M. Pulvirenti,
\textit{Derivation of the Fick’s Law for the Lorentz model in a low density regime},
Commun. Math. Phys. \textbf{336} (2015), 1607--1636.
\bibitem{MeerAnnals}
P. Becker-Kern, M.M. Meerschaert, H.-P. Scheffler,
\textit{Limit theorems for coupled continuous time random walks},
Ann. Probab. \textbf{32} (2004), no. 1B, 730--756.
\bibitem{Beghin} L. Beghin, E. Orsingher, \textit{Fractional Poisson processes and related planar random motions}. Electron.
J. Probab. \textbf{14} (2009), 1790 -- 1826.
\bibitem{Boldrighini}
C. Boldrighini, L.A. Bunimovich, Ya.G. Sinai,
\textit{On the Boltzmann equation for the Lorentz gas},
J. Stat. Phys. \textbf{32} (1983), 477--501.
\bibitem{schillinglevy}
B. Böttcher, R. Schilling, J. Wang,
\textit{Lévy Matters III. Lévy-Type Processes: Construction, Approximation and Sample Path Properties},
Springer, 2013.

\bibitem{bouchaud1990}
Bouchaud, J.-P. and Georges, A.,
Anomalous diffusion in disordered media: Statistical mechanisms, models and physical applications,
\textit{Physics Reports}, \textbf{195}(4-5), 127--293 (1990).

\bibitem{golse1}
J. Bourgain, F. Golse and B. Wennberg,
\textit{On the Distribution of Free Path Lengths for the Periodic Lorentz Gas},
Commun. Math. Phys. \textbf{190} (1998), 491--508.
\bibitem{golse3}
E. Caglioti and F. Golse,
\textit{On the Distribution of Free Path Lengths for the Periodic Lorentz Gas III},
Commun. Math. Phys. \textbf{236} (2003), no. 1, 199--221.
\bibitem{Spizzichino1}
L. Caramellino, F. Spizzichino,
\textit{Dependence and aging properties of lifetimes with Schur-constant survival functions},
Probab. Eng. Inf. Sci. \textbf{8} (1994), no. 1, 103--111.
\bibitem{Daley}
D.J. Daley, D. Vere-Jones,
\textit{An Introduction to the Theory of Point Processes}, Vol. I: Elementary Theory and Methods, 2006; Vol. II: General Theory and Structure,
Springer, 2007.
\bibitem{DesvillettesRicci}
L. Desvillettes, V. Ricci,
\textit{The Boltzmann–Grad limit of a stochastic Lorentz gas in a force field},
Bull. Inst. Math. Acad. Sin. (N.S.) \textbf{2} (2007), no. 2, 637--648.
\bibitem{Dicrescenzo} A. Di Crescenzo, B.  Martinucci and A. Meoli, \textit{A fractional counting process and its connection with the
Poisson process}. ALEA \textbf{13} (2016), 291 -- 307. 
\bibitem{Pagnini}
F. Di Tullio, P. Paradisi, R. Spigler, G. Pagnini,
\textit{Gaussian processes in complex media: new vistas on anomalous diffusion},
Front. Phys. \textbf{7} (2019), 123.
\bibitem{Facciaroni}
L. Facciaroni, C. Ricciuti, E. Scalas, B. Toaldo,
\textit{Para-Markov chains and related non-local equations},
Fract. Calc. Appl. Anal. (2025), 1--23.
\bibitem{fedotov1}
S. Fedotov and N. Korabel,
\textit{Emergence of Levy walks in systems of interacting individuals},
Phys. Rev. E. \textbf{95} (2017), 030107(R).
\bibitem{fedotov2}
S. Fedotov, N. Korabel, T.A. Waigh, D. Han, and V.J. Allan,
\textit{Memory effects and Lévy walk dynamics in intracellular transport of cargoes}, Phys. Rev. E. \textbf{98} (2018), 042136.

\bibitem{Feller}
W. Feller,
\textit{An Introduction to Probability Theory and Its Applications}, Vol. II, 2nd ed.,
Wiley, New York, 1971.
\bibitem{Gallavotti}
G. Gallavotti,
\textit{Rigorous Theory of the Boltzmann Equation in the Lorentz Gas},
Nota interna n. 358, Istituto di Fisica, Università di Roma, 1972.
\bibitem{golse2}
F. Golse and B. Wennberg,
\textit{On the Distribution of Free Path Lengths for the Periodic Lorentz Gas II},
ESAIM M2AN \textbf{34} (2000), 1151-1163.

\bibitem{golding2006}
I.~Golding and E.~C. Cox. Physical nature of bacterial cytoplasm. 
\textit{Physical Review Letters}, 96(9):098102, 2006.


\bibitem{Golse}
F. Golse,
\textit{The Boltzmann–Grad limit for the Lorentz gas with a Poisson distribution of obstacles},
arXiv:2111.15270, 2021.
\bibitem{GorenfloMainardi}
R. Gorenflo, F. Mainardi,
\textit{Fractional Calculus: Integral and Differential Equations of Fractional Order},
in: A. Carpinteri, F. Mainardi (eds.), \textit{Fractals and Fractional Calculus in Continuum Mechanics},
Springer, New York and Wien, 1997, pp. 223--276.


\bibitem{havlin1987}
Havlin, S. and Ben-Avraham, D.,
Diffusion in disordered media,
\textit{Advances in Physics}, \textbf{36}(6), 695--798 (1987).

 \bibitem{isichenko1992}
Isichenko, M. B.,
Percolation, statistical topography, and transport in random media,
\textit{Reviews of Modern Physics}, \textbf{64}(4), 961--1042 (1992).


\bibitem{Lamperti_James}
L.F. James,
\textit{Lamperti-type laws},
Ann. Appl. Probab. \textbf{20} (2010), no. 4, 1303--1340.
\bibitem{barkai1}
Y. Jung, E. Barkai and R.J. Silbey,
\textit{Lineshape theory and photon counting statistics for blinking quantum dots: a Lévy walk process},
Chem. Phys. \textbf{284}(1-2) (2002),  181 -- 194.
\bibitem{kallenberg}
O. Kallenberg,
\textit{Foundations of Modern Probability},
Springer, 1997.
\bibitem{klafter1996}
J. Klafter, M. F. Shlesinger, and G. Zumofen,
\textit{Beyond Brownian Motion},
{Physics Today}, \textbf{49} (1996), no. 2, 33--39.
\bibitem{klaftersokolov}
J. Klafter and I. M. Sokolov,
\textit{First Steps in Random Walks: From Tools to Applications},
Oxford University Press, Oxford, 2011.
\bibitem{kolesnik}
A.D. Kolesnik, E. Orsingher,
\textit{A planar random motion with an infinite number of directions controlled by the damped wave equation},
J. Appl. Probab. \textbf{42} (2005), no. 4, 1168--1182.
\bibitem{Kolobook}
V.N. Kolokoltsov,
\textit{Markov Processes, Semigroups and Generators},
De Gruyter, 2011.
\bibitem{Penrose}
G. Last, M. Penrose,
\textit{Lectures on the Poisson Process}, vol. 7,
Cambridge Univ. Press, 2018.
\bibitem{Lorentz}
H.A. Lorentz,
\textit{The motion of electrons in metallic bodies I},
KNAW, Proc. \textbf{7} (1905), 438--453.

\bibitem{scher1975}
H.~Scher and E.~W. Montroll. Anomalous transit-time dispersion in amorphous solids. 
\textit{Physical Review B}, 12(6):2455--2477, 1975. 
doi:10.1103/PhysRevB.12.2455.


\bibitem{Toth1}
C. Lutsko, B. Tóth,
\textit{Invariance principle for the random Lorentz gas—beyond the Boltzmann–Grad limit},
Commun. Math. Phys. \textbf{379} (2020), no. 2, 589--632.
\bibitem{Toth2}
C. Lutsko, B. Tóth,
\textit{Diffusion of the random Lorentz process in a magnetic field},
arXiv:2411.03984, 2024.
\bibitem{MainardiGorenfloScalas}
F. Mainardi, R. Gorenflo, E. Scalas,
\textit{A fractional generalization of the Poisson processes},
Vietnam J. Math. \textbf{32} (2004), 53--64.
\bibitem{barkai2}
G. Margolin, E. Barkai,
\textit{Nonergodicity of blinking nanocrystals and other Lévy-walk processes},
Physical review letters \textbf{94}  (2005), no. 8, 080601.
\bibitem{Marklof} J. Marklof and A. Str\"ombergsson. \emph{Kinetic theory for the low-density Lorentz gas}, Memoirs of American Mathematical Society \textbf{294} (2024), 1464.
\bibitem{anmath}
J. Marklof and A. Str\"ombergsson,
\textit{The Boltzmann-Grad limit of the periodic Lorentz gas},
Annals of
Mathematics \textbf{174} (2011), 225 -- 298.
\bibitem{toth_super}
J. Marklof and B. T\'oth,
\textit{Superdiffusion in the periodic Lorentz gas},
Comm. Math. Phys. \textbf{347} (2016), 933 -- 981.
\bibitem{meerstraannals}
M.M. Meerschaert and P. Straka,
\textit{Semi-Markov approach to continuous time random walk limit processes},
Ann. Probab. \textbf{42} (2014), no. 4, 1699--1723.
\bibitem{MeerJap}
M.M. Meerschaert, H.-P. Scheffler,
\textit{Limit theorems for continuous-time random walks with infinite mean waiting times},
J. Appl. Probab. \textbf{41} (2004), no. 3, 623--638.
\bibitem{MeerPhysRevE}
M.M. Meerschaert, et al.,
\textit{Governing equations and solutions of anomalous random walk limits},
Phys. Rev. E \textbf{66} (2002), no. 6, 060102.
\bibitem{MeerSpa}
M.M. Meerschaert, H.-P. Scheffler,
\textit{Triangular array limits for continuous time random walks},
Stoch. Process. Their Appl. \textbf{118} (2008), no. 9, 1606--1633.
\bibitem{meertoa}
M.M. Meerschaert, B. Toaldo,
\textit{Relaxation patterns and semi-Markov dynamics},
Stoch. Process. Their Appl. \textbf{129} (2019), no. 8, 2850--2879.
\bibitem{metzler_fd}
R. Metzler, J. Klafter,
\textit{The random walk's guide to anomalous diffusion: a fractional dynamics approach},
Physics Report \textbf{339} (2000), no. 1, 1--77.



\bibitem{metzler2004}
Metzler, R., Barkai, E., and Klafter, J.,
Anomalous diffusion and relaxation close to thermal equilibrium: A fractional kinetics approach,
\textit{Journal of Physics A: Mathematical and General}, \textbf{37}, R161--R208 (2004).

\bibitem{Monin}
A.S. Monin,
\textit{A statistical interpretation of the scattering of microscopic particles},
Theory Probab. Appl. \textbf{1} (1956), 298--310.

\bibitem{montrollshlesinger1984}
Montroll, E. W., and M. F. Shlesinger. "On the wonderful world of random walks." In \textit{Nonequilibrium Phenomena II: From Stochastics to Hydrodynamics}, edito da J. L. Lebowitz e E. W. Montroll, North-Holland, Amsterdam, 1984, pp. 1--121.

\bibitem{NotaSaffirio}
A. Nota, D. Nowak, C. Saffirio,
\textit{Diffusion limit of the low-density magnetic Lorentz gas},
arXiv:2412.11134, 2024.
\bibitem{NotaSaffirioSimonella}
A. Nota, C. Saffirio, S. Simonella,
\textit{Two-dimensional Lorentz process for magnetotransport: Boltzmann–Grad limit},
Ann. Inst. Henri Poincaré Probab. Stat. \textbf{58} (2022), no. 2, 1228--1243.
\bibitem{Orsingher1}
E. Orsingher, A. De Gregorio,
\textit{Random flights in higher spaces},
J. Theor. Probab. \textbf{20} (2007), 769--806.
\bibitem{Orsingher2}
E. Orsingher, C. Ricciuti, F. Sisti,
\textit{Motion among random obstacles on a hyperbolic space},
J. Stat. Phys. \textbf{162} (2016), 869--886.
\bibitem{pinsky}
M.A. Pinsky,
\textit{Isotropic transport process on a Riemannian manifold},
Trans. Amer. Math. Soc. \textbf{218} (1976).

\bibitem{richardson1926}
L.~F. Richardson. Atmospheric diffusion shown on a distance-neighbour graph. 
\textit{Proceedings of the Royal Society of London. Series A}, 110:709--737, 1926.

\bibitem{RicciutiToaldo}
C. Ricciuti, B. Toaldo,
\textit{From semi-Markov random evolutions to scattering transport and superdiffusion},
Commun. Math. Phys. \textbf{401} (2023), no. 3, 2999--3042.
\bibitem{sato}
K. Sato,
\textit{Lévy processes and infinitely divisible distributions},
Cambridge University Press, 1999.
\bibitem{schillingpaper}
R.L. Schilling,
\textit{Subordination in the sense of Bochner and a related functional calculus},
J. Aust. Math. Soc. \textbf{64} (1998), no. 3, 368--396.
\bibitem{schillingBM}
R.L. Schilling, L. Partzsch,
\textit{Brownian motion -- An introduction to stochastic processes}, 2nd ed.,
De Gruyter, Berlin, 2014.
\bibitem{librobern}
R.L. Schilling, R. Song, Z. Vondraček,
\textit{Bernstein functions: theory and applications},
Walter de Gruyter GmbH \& Co. KG, 2012.
\bibitem{Shlesinger1999}
M. F. Shlesinger, J. Klafter, and G. Zumofen,
\textit{Above, below and beyond Brownian motion,}
{Am. J. Phys.}, \textbf{67} (1999), no. 12, 1253--1259.


\bibitem{sokolov2005}
I.~M. Sokolov and J.~Klafter. Diffusion in a system of traps: Sokolov model revisited. 
\textit{Chaos: An Interdisciplinary Journal of Nonlinear Science}, 15(2):026103, 2005.


\bibitem{Spizzichino2}
F. Spizzichino,
\textit{Subjective probability models for lifetimes},
Chapman and Hall/CRC, 2001.
\bibitem{Spohn}
H. Spohn,
\textit{The Lorentz process converges to a random flight},
Commun. Math. Phys. \textbf{60} (1978), 277--290.
\bibitem{Spohn_review}
H. Spohn,
\textit{Kinetic equations from Hamiltonian dynamics: Markovian limits},
Rev. Mod. Phys. \textbf{52}(1980), 569
\bibitem{Stadje}
W. Stadje,
\textit{Exact probability distributions for noncorrelated random walk models},
J. Stat. Phys. \textbf{56} (1989), no. 3/4.


\bibitem{strakaSpa}
P. Straka, B.I. Henry,
\textit{Lagging and leading coupled continuous time random walks, renewal times and their joint limits},
Stoch. Process. Their Appl. \textbf{121} (2011), no. 2, 324--336.

\bibitem{Watanabe}
S. Watanabe, T. Watanabe,
\textit{Convergence of isotropic scattering transport process to Brownian motion},
Nagoya Math. J. \textbf{40} (1970), 161--171.

\bibitem{weiss1994}
G.~H. Weiss. \textit{Aspects and Applications of the Random Walk}. North-Holland, Amsterdam, 1994.


\bibitem{zaburdaev}
V. Zaburdaev, S. Denisov and J. Klafter.
\textit{L\'evy walks}, Reviews of Modern Physics \textbf{87} (2015), no. 2, 483 -- 530.



\bibitem{zumofen1993}
G.~Zumofen and J.~Klafter. Scale-invariant motion in intermittent chaotic systems. 
\textit{Physical Review E}, 47(2):851, 1993.




\end{thebibliography}
\end{document}